\documentclass[11pt]{article}
    \usepackage{url}
    \usepackage{verbatim}
    \usepackage[titletoc]{appendix}
    \usepackage{graphicx}
	\usepackage{amsmath}
	\usepackage{amsthm}
	\usepackage{amsfonts}
	\usepackage{graphicx}
    \usepackage{nicefrac}
    \usepackage{longtable}
    \usepackage{color}
    \usepackage{graphicx, amssymb, subfigure,graphics}
    \usepackage{threeparttable}

    \newcommand{\be}{\begin{equation}}
    \newcommand{\ee}{\end{equation}}

    \newcommand{\nrm}[1]{\left\| #1 \right\|}

    \newcommand\dt {{\Delta t}}

    \def\0{\mbox{\boldmath $0$}}

    \def\g{\mbox{\boldmath $g$}}
    \def\h{\mbox{\boldmath $h$}}

    \def\0{\mbox{\boldmath $0$}}

	\newtheorem{thm}{Theorem}[section]

	\newtheorem{defn}[thm]{Definition}
	\newtheorem{lem}[thm]{Lemma}
	\newtheorem{rem}[thm]{Remark}

\begin{document}
	\title{A second order accurate numerical scheme for the porous medium equation by an energetic variational approach}
	
	\author{
Chenghua Duan \thanks{Department of Mathematics, Soochow University, Suzhou 215006, China and Shanghai Center for Mathematical Sciences, Fudan University, Shanghai 200438, China (chduan@fudan.edu.cn)}
\and Wenbin Chen \thanks{School of Mathematical Sciences, Fudan University, Shanghai 200438, China (wbchen@fudan.edu.cn)}
\and
Chun Liu \thanks{Department of Applied Mathematics, Illinois Institute of Technology, Chicago, IL 60616, USA (cliu124@iit.edu)}
\and		
Cheng Wang\thanks{Department of Mathematics, The University of Massachusetts, North Dartmouth, MA  02747 (Corresponding Author: cwang1@umassd.edu)}	
\and
Xingye Yue\thanks{Department of Mathematics, Soochow University, Suzhou 215006, China (xyyue@suda.edu.cn)}
}	
	
\date{}
\maketitle
\numberwithin{equation}{section}

{
\begin{abstract}
\footnotesize
The porous medium equation (PME) is a typical nonlinear degenerate parabolic equation. 
An energetic variational approach has been studied in a recent work~\cite{C. Duan(2018)}, in which the trajectory equation is obtained, and a few first order accurate numerical schemes have been developed and analyzed. 
In this paper,  we construct and analyze a second order accurate numerical scheme in both time and space. The unique solvability, energy stability are established, based on the convexity analysis. In addition, we provide a detailed convergence analysis for the proposed numerical scheme. A careful higher order asymptotic expansion is performed and two step error estimates are undertaken. In more details, a rough estimate is needed to control the highly nonlinear term in a discrete $W^{1,\infty}$ norm, and a refined estimate is applied to derive the optimal error order. Some numerical examples are presented as well.


 \end{abstract}
}

\section{Introduction and background }
\label{sec:1}
In this paper,  we consider the porous medium equation (PME):  $$ \partial_t f=\Delta_{x} (f^{m}), \ x\in\Omega\subset\mathbb{R}^d,\  m>1,$$
where  $f:=f(x,t)$ is a non-negative scalar function of space $x\in\mathbb{R}^d$ ($d \ge 1$) and the time $t\in\mathbb{R}^{+}$, and $m$ is a constant larger than 1.  It has been widely  applied   in many physical and biological models, such as an isentropic gas flow through a porous medium,  the viscous gravity currents, nonlinear heat transfer and image processing \cite{J. L. Vazquez(2007)}, etc.  

One basic characteristic of the PME is its degeneracy at points where $f = 0$.
In turn, there are  many special features: the finite speed of  propagation, the free boundary, a possible waiting time phenomenon \cite{D.G.Aronson(1983), C. Duan(2018), J. L. Vazquez(2007)}.
Many theoretical analyses have been derived in the existing literature \cite{D.G. Aronson(1969), A.S. Kalasnikov(1967), O.A. Oleinik(1958), S. Shmarev(2003),S. Shmarev(2005), J. L. Vazquez(2007)}, etc.     Meanwhile, various numerical methods have been studied for the PME, such as  finite difference approach \cite{J. L. Graveleau(1971)}, tracking algorithm method \cite{E. DiBenedetto(1984)},  a  local discontinuous Galerkin finite element method \cite{Q. Zhang(2009)},  Variational Particle Scheme (VPS) \cite{M. Westdickenberg(2010)} and  an adaptive moving mesh finite element method \cite{C. Ngo(2017)}.  

The numerical methods have also been developed for  the PME  by an Energetic Variational Approach (EnVarA).  In this way, the numerical solution can naturally keep the physical laws, including  the conservation of mass, energy dissipation and force balance.    Based on different dissipative energy laws, the numerical schemes have been studied in  \cite{C. Duan(2018)}. 
Besides a good approximation for the solution without oscillation and the free boundary, the notable advantage is that the waiting time problem could be naturally treated, which has been a well-known difficult issue for all the existing methods.  In addition,  the second order convergence in space and the first order convergence in time have been reported for both schemes in \cite{C. Duan(2018),C.H.Duan(2019)}.  

The aim of the paper is to construct a second order scheme in both time and space, which is also uniquely solvable  and satisfies the discrete energy dissipation law on the admission set at theoretical level.  In addition, we provide an optimal rate convergence analysis for the proposed second order numerical scheme. In particular, the highly nonlinear nature of  the trajectory equation  makes the problem very challenging. To overcome these subtle difficulties, we use a higher order expansion technique to ensure a higher order consistency estimate, which is needed to obtain a discrete $W^{1,\infty}$ bound of the numerical solution. Similar ideas have been reported in earlier literature for incompressible fluid equations~\cite{C.H.Duan(2019), W. E(1995), W. E(2002), R. Samelson(2003), C. Wang(2000)}, while the analysis presented in this work turns out to be more complicated, due to the lack of a linear diffusion term in the trajectory equation of PME and the  requirement of the high order numerical scheme. In addition, to recover the nonlinear analysis, we have to carry out two step estimates:  rough estimate and refined estimate  \cite{C.H.Duan(2019)}. 
Different from a standard error estimate, the rough estimate controls the nonlinear term, which is an effective approach to handle the highly nonlinear term.

   This paper is organized as follows. The trajectory equation of  PME and  the numerical scheme are outlined in Sections \ref{sec:trajectory} \ref{sec:num}, respectively.  The proof of unique solvability analysis, unconditional energy stability, and optimal rate convergence analysis  are  provided in Sections \ref{sec:uni}, \ref{sec:ene} and \ref{sec:con}, respectively. Moreover, the convergence analysis of  Newton's iteration for the nonlinear numerical scheme can be found in Section  \ref{sec:Newton}. Finally we present a simple numerical example to demonstrate the convergence rate of the numerical scheme in Section \ref{sec:numResults}. Some concluding remarks are made in Section~\ref{sec: conclusion}.    

\section{Trajectory equation of the PME}\label{sec:trajectory}
 In this part,  the one-dimensional trajectory equation will be reviewed, derived by an Energetic Variational Approach. We solve  the following initial-boundary problem:
\begin{eqnarray}
&&
  \partial_t f+\partial_{x} (f{\bf v})=0, \ x\in\Omega\subset\mathbb{R}, \ t>0,
 \label{eqcm}\\
&& f{\bf v}=-\partial_{x}(f^{m}), \ x\in\Omega, \  m>1, \label{eqDa}\\
&& f(x,0)=f_{0}(x)\geq 0, \ x\in\Omega, \label{eqini} \\
&&\partial_{x}f=0, \ x\in\partial\Omega,\ t>0, 
 \label{eqbc}
 \end{eqnarray}
where $f$ is a non-negative function,  $\Omega$ is a bounded domain and  ${\bf v}$ is the velocity.

The following lemma is available.

\begin{lem} \label{lem-dissip}
  $f(x,t)$ is a positive solution of \eqref{eqcm}-\eqref{eqbc} if and only if
  $f(x,t)$ satisfies the corresponding energy dissipation law:
  \begin{equation}\label{equ:energylaw}
   \frac{d}{dt}\int_{\Omega}f\ln fdx=-\int_{\Omega}\frac{f}{mf^{m-1}}|\textbf{v}|^{2}dx.
  \end{equation}
    \end{lem}
The detail proof of Lemma  \ref{lem-dissip} can be found in \cite{C.H.Duan(2019)}.

\begin{rem}
There is an assumption that the value of initial state $f_0(x)$ is positive  in $\Omega$ to  make $\int_{\Omega}f\ln fdx$ well-defined in \eqref{equ:energylaw}.\end{rem}


 Based on the Energetic Variational Approach, we obtain the initial-boundary problem of  trajectory  $x$ in the Lagrangian coordinate:

 \begin{eqnarray}
 &&    \frac{f_{0}(X)}{m\big(\frac{f_0(X)}{\partial_X x}\big)^{m-1}}\partial_{t} {x}=-\partial_{X}\left(\frac{f_0(X)}{\partial_X x}\right),\ X\in\Omega,\ t>0,  \label{eqtra}\\
 && x|_{\partial\Omega}= X|_{\partial\Omega},\ t>0,\label{eqtrabou} \\
     && x(X,0)=X, \ X\in\Omega. \label{eqtraini}
    \end{eqnarray}
    Note that (\ref{eqcm}) is the conservation law. In the Lagrangian coordinate, its solution can be expressed by:
 \begin{equation}\label{equ:conservationL}
  f(x(X,t),t)=\frac{f_{0}(X)}{\frac{\partial x(X,t)}{\partial X}},
\end{equation}
where $f_{0}(X)$ is the positive initial data and $\partial_X x:=\frac{\partial x(X,t)}{\partial X}$ is the \emph{deformation\ gradient} in one dimension.

Finally, with a substitution of \eqref{eqtra}-\eqref{eqtraini} into \eqref{equ:conservationL}, we obtain the solution $f(x,t)$ to \eqref{eqcm}-\eqref{eqbc}.

\section{The second order accurate numerical scheme} \label{sec:num}
%
    Let $X_0$ be the left point of $\Omega$ and $h=\frac{|\Omega|}{M}$  be the spatial step,  $M\in\mathbb{N}^{+}$. Denote by $X_{r}=X(r)=X_0+ r h$, where $r$ takes on integer and half integer values.  Let $\mathcal{E}_{M}$ and $\mathcal{C}_{M}$ be the spaces of functions whose domains are $\{X_{i}\ |\ i=0,...,M\}$ and $\{X_{i-\frac{1}{2}}\ |\ i=1,...,M\}$, respectively. In component form, these functions are identified via
$l_{i}=l(X_{i})$, $i=0,...,M$, for $l\in\mathcal{E}_{M}$, and $\phi_{i-\frac{1}{2}}=\phi(X_{i-\frac{1}{2}})$,  $i=1,...,M$, for $\phi\in\mathcal{C}_{M}$.

\indent{The} difference operator $D_{h}: \mathcal{E}_{M}\rightarrow\mathcal{C}_{M}$, $d_{h}: \mathcal{C}_{M}\rightarrow\mathcal{E}_{M}$, and $\widetilde{D}_h: \mathcal{E}_{M}\rightarrow\mathcal{E}_{M}$  can be defined  as:
\begin{align}\label{equ:dif1}
& (D_{h}l)_{i-\frac{1}{2}}= (l_{i}-l_{i-1})/h,\ i=1,...,M, \\
&  (d_{h}\phi)_{i}= (\phi_{i+\frac{1}{2}}-\phi_{i-\frac{1}{2}})/h,\ i=1,...,M-1,\\
&(\widetilde{D}_hl)_{i}=\left\{
\begin{array}{lcl}
(l_{i+1}-l_{i-1})/2h, &\mbox{$ i=1,...,M-1$},\\
 (4l_{i+1}-l_{i+2}-3l_{i})/2h,&  i=0,  \\
 (l_{i-2}-4l_{i-1}+3l_{i})/2h, &  i=M.\end{array}\right.
\end{align}

Let $\mathcal{Q}:=\{l \in\mathcal{E}_{M}\ |\ l_{i-1}<l_{i},\ 1\leq i\leq M;\ l_{0}=X_0,\ l_{M}=X_M\}$ be the admissible set, in which the particles are arranged in the order without twisting or exchanging. Its boundary set is $\partial\mathcal{Q}:=\{l \in\mathcal{E}_{M}\ |\ l_{i-1}\leq l_i,\ 1\leq i\leq M,\ and\  l_{i}=l_{i-1},\ for\ some\ 1\leq i\leq M;\ l_{0}=X_0,\ l_{M}=X_M\}$. Then $\bar{\mathcal{Q}}:=\mathcal{Q}\cup\partial\mathcal{Q}$ is a closed convex set

We propose the second order numerical scheme as follows, based on a modified Crank-Nicolson approach. Given the positive initial state $f_0(X)\in\mathcal{E}_M$ and the particle position $x^n, x^{n-1} \in\mathcal{Q}$, find $x^{n+1}=(x^{n+1}_{0},...,x^{n+1}_{M})\in \mathcal{Q}$ such that
\begin{eqnarray}
   \frac{f_{0}(X_{i})}{m\big(\frac{f_0(X)}{ S_h (x^n , x^{n-1} ) } \big)_{i}^{m-1}}\cdot\frac{x^{n+1}_{i}-x^{n}_{i}}{\dt}  &=&
  -d_{h}\Big[\Big( f_{0}(X) \frac{\ln (D_h x^{n+1}) - \ln ( D_h x^n) }{D_h x^{n+1} - D_h x^n} \Big) \nonumber
\\
  && - A_0 \dt D_h ( x^{n+1} - x^n ) + \dt^2 ( \frac{1}{D_h x^{n+1}} - \frac{1}{D_h x^n} ) \Big]_{i} ,  \label{scheme-2nd-1}
\\
  &&
  \mbox{with} \quad
  S_h (x^n, x^{n-1}) = \max ( \widetilde{D}_h ( \frac32 x^n - \frac12 x^{n-1} ) , \dt^2 ) ,
  \nonumber
\end{eqnarray}
for $1\leq i \leq M-1$, and we take with $x_0^{n+1}=0$ and $x_M^{n+1}=1$, $n=0,\cdots,N-1$.

To solve the nonlinear equation  \eqref{scheme-2nd-1}, we use  Damped  Newton's iteration \cite{Y. Nesterov(1994)}.

\noindent{\bf  Damped Newton's iteration}. \ \ Set $x^{n+1, 0}= x^n$. For $k=0,1,2,\cdots$, update $x^{n+1,k+1}=x^{n+1,k} +  \omega(\lambda)\delta_x$ such that
\begin{eqnarray}\label{equ:numnum1}
&&
\frac{f_{0}(X_{i})}{m\big(\frac{f_0(X)}{ S_h (x^n , x^{n-1} ) } \big)_{i}^{m-1}}\cdot\frac{ \delta_{x_{i}}}{\tau} - d_h\Big[ \Big(f_{0}(X) W^{n+1,k}+A_0\tau+\frac{\tau^2}{(D_h x^{n+1,k})^2}\Big) D_h \delta_{{x}}\Big]_{i}\nonumber\\
&&
= - \frac{f_{0}(X_{i})}{m\big(\frac{f_0(X)}{ S_h (x^n , x^{n-1} ) } \big)_{i}^{m-1}}\frac{x^{n+1,k}_{i} -x^{n}_{i}}{\tau} - d_{h}\Big[ f_{0}(X) R^{n+1,k} \nonumber
\\
  &&\ \ \   - A_0 \tau D_h ( x^{n+1,k} - x^n ) + \tau^2 ( \frac{1}{D_h x^{n+1,k}} - \frac{1}{D_h x^n} ) \Big]_{i} ,  \label{scheme-2nd-Newton}
\ \ 1\leq i\leq M-1,\\
&&
\delta_{x_0}=\delta_{x_M}=0,  \nonumber
\end{eqnarray}
 where for $i=1,\cdots,M$, 
\begin{equation*}
W_{i-\frac{1}{2}}^{n+1,k}=\left\{
\begin{array}{lcl}
\left[\frac{(1-\frac{D_h x^{n}}{D_h x^{n+1,k}})+\ln(\frac{D_h x^{n}}{D_h x^{n+1,k}})}{(D_h x^{n+1,k}-D_h x^n)^2}\right]_{i-\frac{1}{2}}, & |D_h x^{n+1,k}_{i-\frac{1}{2}}-D_h x^n_{i-\frac{1}{2}}|\neq 0,\\
-\frac{1}{2\Big(D_h x^{n+1,k}_{i-\frac{1}{2}}\Big)^2}, & |D_h x^{n+1,k}_{i-\frac{1}{2}}-D_h x^n_{i-\frac{1}{2}}|= 0,
\end{array}\right.
\end{equation*}

\begin{equation*}
R^{n+1,k}_{i-\frac{1}{2}}=\left\{
\begin{array}{lcl}
\Big[\frac{\ln (D_h x^{n+1,k}) - \ln ( D_h x^n) }{D_h x^{n+1,k} - D_h x^n}\Big]_{i-\frac{1}{2}}, & |D_h x^{n+1,k}_{i-\frac{1}{2}}-D_h x^n_{i-\frac{1}{2}}|\neq 0,\\
\frac{1}{D_h x^{n+1,k}_{i-\frac{1}{2}}}, & |D_h x^{n+1,k}_{i-\frac{1}{2}}-D_h x^n_{i-\frac{1}{2}}|= 0.
\end{array}\right.
\end{equation*}
Let  $\hat{x} := x^{n+1} - X$ and  \begin{equation}\label{Newton}
\omega(\lambda)=\left\{
\begin{array}{lcl}
\frac{1}{\lambda}, &&{\lambda > \lambda'},\\
\frac{1-\lambda}{\lambda(3-\lambda)}, &&{\lambda'\ge\lambda\ge\lambda^{*}},\\
1, && {\lambda < \lambda^{*}},
\end{array}\right.
\end{equation}
where $\lambda^*=2-3^{\frac{1}{2}}$, $\lambda'\in[\lambda^*, 1)$ and 
\begin{equation}\label{lambda}
\lambda^2:=\lambda^2(F,\hat{x}^{n+1,k})=\frac{1}{a}[F'(\hat{x}^{n+1,k})]^T[F''(\hat{x}^{n+1,k})]^{-1}F'(\hat{x}^{n+1,k}),
\end{equation}
with $a:=(h\min\limits_{0\leq i\leq M}f_0(X_i))/2C_{Newton}^2$ (a  constant $C_{Newton}>0$),  the corresponding discrete  functional  $F(\hat{x})$  defined latter in \eqref{solvability-2-1}-\eqref{solvability-2-4},  its   gradient vector $F'$ and Hessian matrix $ F''$.

Then we obtain the numerical solution $f(x^n,t^n):=f^n_i$   by
\begin{equation}\label{eqDen}
f^n_i=\frac{f_0(X)}{\widetilde{D}_h x^n_i},\ 0\leq i\leq M,
\end{equation}
which is the discrete scheme of \eqref{equ:conservationL}.

\section{Unique solvability analysis}\label{sec:uni}

\begin{thm}  \label{thm:solvability}
Given any $x^n, x^{n-1}$, with
\begin{equation}
  0 < Q^{(k),1}  \le D_h x^k \le Q^{(k),2} ,   \quad \mbox{for $k=n, n-1$} .
  \label{assumption-1}
\end{equation}
The proposed numerical scheme~\eqref{scheme-2nd-1} is uniquely solvable, with $D_h x^{n+1} > 0$ at a point-wise level.
\end{thm}

\begin{proof}
With an introduction $\hat{x} = x^{n+1} - X$, it is clear that \eqref{scheme-2nd-1} could be rewritten as
\begin{eqnarray}
   \frac{f_{0}(X_{i})}{m\big(\frac{f_0(X)}{ S_h (x^n , x^{n-1} ) } \big)_{i}^{m-1}}\cdot\frac{X_i + \hat{x}_{i}-x^{n}_{i}}{\dt}  &=&
  -d_{h}\Big[\Big( f_{0}(X) \frac{\ln ( 1+ D_h \hat{x}) - \ln ( D_h x^n) }{1 + D_h \hat{x} - D_h x^n} \Big) \nonumber
\\
  && - A_0 \dt D_h ( \hat{x} - x^n ) + \dt^2 ( \frac{1}{1+ D_h \hat{x}} - \frac{1}{D_h x^n} ) \Big]_{i} .  \label{scheme-2nd-alt-1}
\end{eqnarray}
Because of the fact that $\hat{x}_0 =\hat{x}_M=0$, we see that the solution of~\eqref{scheme-2nd-alt-1} is equivalent to a minimization of the following discrete functional:
\begin{eqnarray}
  F (\hat{x}) = \sum_{j=1}^4 F_j (\hat{x}) ,  \quad \mbox{with}  &&
  F_1 (\hat{x}) = \frac{1}{2 \dt} \Bigl\langle \frac{f_{0}(X)}{m\big(\frac{f_0(X)}{ S_h (x^n , x^{n-1} ) } \big)^{m-1}} , ( X + \hat{x} - x^n )^2 \Bigr\rangle  ,   \label{solvability-2-1}
\\
  &&
  F_2 (\hat{x}) = \Big\langle f_0 (X) G (D_h \hat{x} , D_h x^n ) , {\bf 1} \Big\rangle  ,
  \label{solvability-2-2}
\\
  &&
  F_3 (\hat{x}) = A_0 \dt \Bigl( \frac12 \| D_h \hat{x} \|_2^2
  - \langle D_h \hat{x}, D_h x^n \rangle \Bigr) ,   \label{solvability-2-3}
\\
  &&
  F_4 (\hat{x}) = \dt^2 \Bigl( - \langle \ln (1 + D_h \hat{x} , {\bf 1} \rangle
  + \langle D_h \hat{x}, \frac{1}{D_h x^n} \rangle \Bigr)  ,   \label{solvability-2-4}
\end{eqnarray}
in which $G (x,x_0)$ is given by the primitive function of $- \frac{\ln (1+x) - \ln x_0}{1+ x- x_0}$, for a fixed $x_0$:
\begin{equation}
  G (x, x_0) = \int_x^0  \frac{\ln (1+t) - \ln x_0}{1+t- x_0}  \, d t ,
  \quad  \mbox{for $x \ge -1$} .
  \label{solvability-2-2-b}
\end{equation}
The convexity of $F_1$, $F_3$ and $F_4$ (in terms of $\hat{x}$) is obvious. For the functional $F_2$, we have the following observation, for $x> -1$:
\begin{eqnarray}
  G'' (x, x_0) = \Bigl( - \frac{\ln (1+x) - \ln x_0}{1+x- x_0} \Bigr)'_x
  = \frac{-\frac{1}{1+x} (1+x- x_0) + ( \ln (1+x) - \ln x_0)}{(1+x- x_0)^2}
  \ge 0 ,  \label{solvability-2-2-c}
\end{eqnarray}
in which the convexity of $-\ln (1+x)$ has been used. This fact implies the convexity of $F_2$. Therefore, we conclude that $F$ is convex in terms of $\hat{x}$, provided that $D_h \hat{x} > -1$ at a point-wise level. Furthermore, $F$ is strictly convex, because of the strict convexity of $F_1$.

In the next step, we wish to prove that there exists a minimizer of $F$ at an interior point of $\mathcal{Q}$. To this end, consider the following closed domain: for a given $\delta > 0$,
	\begin{equation}
\mathcal{Q}_{\delta} := \left\{ X + \hat{x} \in \mathcal{Q} \ : \  1 + ( D_h \hat{x} )_{i+\nicefrac12} \ge \delta, \forall 0 \le i \le M-1 \right\} \subset \mathcal{Q} .
	\label{positivity-1}
	\end{equation}
Since $\mathcal{Q}_{\delta}$ is a compact and convex set in $\mathbb{R}^{M-1}$, there exists a (not necessarily unique) minimizer of $F$ over $\mathcal{Q}_{\delta}$. The key point of our positivity analysis is that such a minimizer could not occur on the boundary of $\mathcal{Q}_{\delta}$, if $\delta$ is small enough.

Assume a minimizer of $F$ over $\mathcal{Q}_{\delta}$, denote it by $\hat{x}^\star$, occurs at a boundary point. There is at least one grid point such that $1 + ( D_h \hat{x}^* )_{i_0 +\nicefrac12} = \delta$. Next we estimate the value of $F (\hat{x}^*)$. For the $F_1$ part, the following bound is available, for any $X + \hat{x} \in \mathcal{Q}$:
\begin{eqnarray}
  0 \le F_1 (\hat{x}) &=& \frac{1}{2 \dt} \Bigl\langle \frac{f_{0}(X)}{m\big(\frac{f_0(X)}{ S_h (x^n , x^{n-1} ) } \big)^{m-1}} , ( X + \hat{x} - x^n )^2 \Bigr\rangle  \nonumber
\\
  &\le&
   \frac{1}{2 m \dt} \Bigl( \tilde{C}_1^{|m-2|} ( \frac32 Q^{(n),2}
    + \frac12 Q^{(n-1),2} )^{m-1} \Bigr) := A^{(1)} ,
  \label{positivity-2-1}
\\
  \mbox{with} &&
  \tilde{C}_1 = \max \Bigl( \max_\Omega f_0 (X) , \frac{1}{\min_\Omega f_0 (X)} \Bigr) ,
  \nonumber
\end{eqnarray}
in which the assumption~\eqref{assumption-1} has been recalled, and we have made use of the following fact:
\begin{eqnarray}
  0 \le X + \hat{x} \le 1 ,  \, \, \, \mbox{so that} \, \, \,
   -1 \le X + \hat{x} - x^n \le 1 ,   \quad \mbox{at a point-wise level}.
   \label{positivity-2-2}
\end{eqnarray}
For the $F_2$ part, we observe that $G (x, x_0) \ge 0$ for $-1 \le x \le 0$, and
\begin{eqnarray}
  G (x, x_0) = - \int_0^x  \frac{\ln (1+t) - \ln x_0}{1+t- x_0}  \, d t
  \ge - \int_0^x  \frac{1}{1+t}  \, d t   = - \ln (1+x) ,
  \quad  \mbox{for $x \ge 0$} ,
  \label{positivity-3-1}
\end{eqnarray}
in which the convexity of $\ln (1+t)$ has been applied. Meanwhile, by the fact that $X _+ \hat{x} \in \mathcal{Q}$, we have the following observation:
\begin{equation}
  0 < 1 + ( D_h \hat{x} )_{i+\nicefrac12} \le \frac{1}{h} ,  \, \, \, \forall 0 \le i \le M-1 , \quad
  \mbox{since $0 \le x_i \le 1$, $0 \le x_{i+1} \le 1$} .
  \label{positivity-3-2}
\end{equation}
In turn, its substitution into~\eqref{positivity-3-1} implies that
\begin{eqnarray}
  G ( D_h \hat{x} , D_h x^n ) \ge - \ln \frac{1}{h} = \ln h ,
  \quad  \mbox{at any grid point} .
  \label{positivity-3-3}
\end{eqnarray}
As a consequence, we obtain a lower bound for $F_2$:
\begin{eqnarray}
  F_2 (\hat{x}^*) = \Big\langle f_0 (X) G (D_h \hat{x}^* , D_h x^n ) , {\bf 1} \Big\rangle
  \ge \| f_0 (X) \|_\infty \cdot \ln h ,
  \label{positivity-3-4}
\end{eqnarray}
The derivation for a lower bound of $F_3$ is straightforward:
\begin{eqnarray}
  F_3 (\hat{x}^*) \ge - A_0 \dt \langle D_h \hat{x}, D_h x^n \rangle
  \ge - \frac{A_0 \dt}{h^2}  ,
  \label{positivity-4-1}
\end{eqnarray}
in which the inequality~\eqref{positivity-3-2} has been applied. For the functional $F_4$, we see that the second part has the following lower bound:
\begin{eqnarray}
  \dt^2 \langle D_h \hat{x}, \frac{1}{D_h x^n} \rangle
  \ge - \dt^2 \cdot \frac{1}{h} \cdot \frac{1}{Q^{(n),1} }
  =  - \frac{\dt^2}{Q^{(n),1} h}  ,
  \label{positivity-5-1}
\end{eqnarray}
in which the inequality~\eqref{positivity-3-2} and the assumption~\eqref{assumption-1} have been used. For the first part of $F_4$, we recall that $1 + ( D_h \hat{x}^* )_{i_0 +\nicefrac12} = \delta$, and the following estimate is available:
\begin{eqnarray}
  - \langle \ln (1 + D_h \hat{x}^* , {\bf 1} \rangle
  &=& - h \Bigl(  \ln ( 1 + ( D_h \hat{x}^* )_{i_0 +\nicefrac12} )
  + \sum_{i \ne i_0}  \ln ( 1 + ( D_h \hat{x}^* )_{i +\nicefrac12} )  \Bigr)   \nonumber
\\
  &\ge&
  - h \Bigl(  \ln \delta + (M-1) \ln \frac{1}{h}  \Bigr)
  = h \Bigl(  \ln \frac{1}{\delta} + (M-1) \ln h  \Bigr)    \nonumber
\\
  &\ge&
  h \ln \frac{1}{\delta} + \ln h ,   \label{positivity-5-2}
\end{eqnarray}
in which the inequality~\eqref{positivity-3-2} has been applied in the second step, and we have use the fact that $h \cdot M =1$ in the last step. In turn, we get a lower bound for $F_4 (\hat{x}^*)$:
\begin{eqnarray}
  F_4 (\hat{x}^*) =  \dt^2 \Bigl( - \langle \ln (1 + D_h \hat{x} , {\bf 1} \rangle
  + \langle D_h \hat{x}, \frac{1}{D_h x^n} \rangle \Bigr)
  \ge \dt^2 h \ln \frac{1}{\delta}
  - \frac{\dt^2}{Q^{(n),1} h}  + \dt^2 \ln h .   \label{positivity-5-3}
\end{eqnarray}
Therefore, a combination of~\eqref{positivity-2-1},  \eqref{positivity-3-4}, \eqref{positivity-4-1} and ~\eqref{positivity-5-3} yields a lower bound for $F (\hat{x}^*)$:
\begin{eqnarray}
  F (\hat{x}^*)
  \ge \dt^2 h \ln \frac{1}{\delta} - A_{\dt,h} ,  \quad \mbox{with} \, \, \,
  A_{\dt,h} = \frac{A_0 \dt}{h^2} + \frac{\dt^2}{Q^{(n),1} h}
  - ( \dt^2 + \| f_0 (X) \|_\infty) \ln h .   \label{positivity-6}
\end{eqnarray}

Meanwhile, we observe that, by taking $\hat{x}^0 = \0$, so that $X + \hat{x}^0 \in \mathcal{Q}_{\delta}$, the following estimates are available:
\begin{eqnarray}
  0 \le F_1 (\hat{x}^0)
  \le  A^{(1)} ,   \, \, \, \mbox{(by~\eqref{positivity-2-1})} ,  \quad
  F_2 (\hat{x}^0) = 0 ,  \, \, \,  F_3 (\hat{x}^0) = 0 ,  \, \, \,
   F_4 (\hat{x}^0) = 0  ,   \label{positivity-7-1}
\end{eqnarray}
so that
\begin{eqnarray}
  0 \le F (\hat{x}^0) \le A^{(1)} .   \label{positivity-7-2}
\end{eqnarray}
We also notice that both $A_{\dt,h}$ and $A^{(1)}$ are independent of $\delta$. Consequently, by taking $\delta >0$ sufficiently small so that
\begin{eqnarray}
  \dt^2 h \ln \frac{1}{\delta} - A_{\dt,h}  > A^{(1)} ,   \quad \mbox{i.e.} \, \, \,
  0 < \delta < \exp \Bigl( - \frac{A_{\dt, h} + A^{(1)} }{\dt^2 h} \Bigr) .
    \label{positivity-7-3}
\end{eqnarray}
This yields a contradiction that $F$ takes a global minimum at $\hat{x}^\star$ over $\mathcal{Q}_{\delta}$, because $  F (\hat{x}^\star) > F (\hat{x}^0)$. As a result, the global minimum of $F$ over $\mathcal{Q}_{\delta}$ could only possibly occur at an interior point, with $\delta$ satisfying~\eqref{positivity-7-3}. We conclude that there must be a solution $\hat{x} \in \left( \mathcal{Q}_{\delta} \right)^{\mathrm{o}}$, the interior region of $\mathcal{Q}_{\delta}$, so that for all $\psi\in\mathcal{C}_{\rm per}$,
	\begin{equation}
0 =  d_s F (\hat{x}+s\psi) |_{s=0} ,
	\label{positive-8}
	\end{equation}
which is equivalent to the numerical solution of \eqref{scheme-2nd-alt-1}, provided that $\delta$ satisfies~\eqref{positivity-7-3}. The existence of a numerical solution of~\eqref{scheme-2nd-1}, with ``positive" gradient, is established. In addition, since $F$ is a strictly convex function over $\mathcal{Q}$, the uniqueness analysis for this numerical solution is straightforward.
\end{proof}

\section{Unconditional energy stability}\label{sec:ene}

\begin{thm}  \label{thm:energy stability}
The proposed numerical scheme~\eqref{scheme-2nd-1} is unconditionally energy stable: $E_h (x^{n+1}) \le E_h (x^n)$, with $E_h (x^k) := \langle f_{0}(X) \ln ( D_h x^n) ,  {\bf 1} \rangle$.   
\end{thm}

\begin{proof}
Taking a discrete inner product with~\eqref{scheme-2nd-1} by $x^{n+1} - x^n$, making use of the summation by parts formula (because of the boundary condition $(x^{n+1} - x^n)_0 = (x^{n+1} - x^n)_M =0$), we get
\begin{eqnarray}
  &&
   \frac{1}{\dt} \Big\langle \frac{ ( S_h (x^n , x^{n-1} ) )^{m-1} }{m ( f_0(X) )^{m-2} } ,
   (x^{n+1} -x^n )^2 \Big\rangle
  - \Big\langle f_{0}(X) \frac{\ln (D_h x^{n+1}) - \ln ( D_h x^n) }{D_h x^{n+1} - D_h x^n} ,
  D_h x^{n+1} - D_h x^n \Bigr\rangle \nonumber
\\
  &&
  + A_0 \dt \| D_h ( x^{n+1} - x^n ) \|_2^2
  - \dt^2 \Big\langle \frac{1}{D_h x^{n+1}} - \frac{1}{D_h x^n} ,
   D_h x^{n+1} - D_h x^n \Big\rangle = 0 .   \label{eng stab-1}
\end{eqnarray}
The first term on the left hand side turns out to be non-negative, since $S_h (x^n , x^{n-1} ) >0$, $f_0 (X) >0$ at a point-wise level:
\begin{eqnarray}
  \Big\langle \frac{ ( S (x^n , x^{n-1} ) )^{m-1} }{m ( f_0(X) )^{m-2} } ,
   (x^{n+1} -x^n )^2 \Big\rangle  \ge 0 .  \label{eng stab-2}
\end{eqnarray}
The second term exactly gives the difference between the discrete energy values at time steps $t^{n+1}$ and $t^n$:
\begin{eqnarray}
  &&
  - \Big\langle f_{0}(X) \frac{\ln (D_h x^{n+1}) - \ln ( D_h x^n) }{D_h x^{n+1} - D_h x^n} ,
  D_h x^{n+1} - D_h x^n \Bigr\rangle   \nonumber
\\
  &=&
  - \Big\langle f_{0}(X) \ln (D_h x^{n+1}) ,  {\bf 1} \Bigr\rangle
  + \Big\langle f_{0}(X) \ln ( D_h x^n) ,  {\bf 1} \Bigr\rangle  
   = E_h (x^{n+1}) - E_h (x^n) . \label{eng stab-3}
\end{eqnarray}
The third term is clearly non-negative, and the last term turns out to be non-negative as well:
\begin{eqnarray}
  - \Big\langle \frac{1}{D_h x^{n+1}} - \frac{1}{D_h x^n} ,
   D_h x^{n+1} - D_h x^n \Big\rangle
   = \Big\langle \frac{1}{D_h x^{n+1} D_h x^n} ,
   ( D_h x^{n+1} - D_h x^n )^2 \Big\rangle  \ge 0 ,   \label{eng stab-4}
\end{eqnarray}
in which we have made use of the unique solvability result, $D_h x^{n+1} > 0$, $D_h x^n > 0$, at the point-wise level, as given by Theorem~\ref{thm:solvability}. As a consequence, a substitution of~\eqref{eng stab-2}-\eqref{eng stab-4} into \eqref{eng stab-1} reveals an unconditional energy stability of the numerical scheme:
\begin{eqnarray}
   E_h (x^{n+1}) - E_h (x^n)  \le - A_0 \dt \| D_h ( x^{n+1} - x^n ) \|_2^2  \le 0 ,  \quad
   \mbox{so that $E_h (x^{n+1}) \le E_h (x^n)$} .
\end{eqnarray}
This completes the proof of Theorem~\ref{thm:energy stability}.
\end{proof}

\section{Optimal rate convergence analysis}\label{sec:con}

Now we proceed into the convergence analysis. Let $x_e$ be the exact solution for the PME equation~\eqref{eqtra}-\eqref{eqtraini}. With sufficiently regular initial data, we could assume that the exact solution has regularity of class $\mathcal{R}$:
\begin{equation}
x_e \in \mathcal{R} := H^6 \left(0,T; C (\Omega)\right) \cap H^4 \left(0,T; C^2 (\Omega)\right) \cap L^\infty \left(0,T; C^6 (\Omega)\right).
	\label{assumption:regularity.1}
	\end{equation}
In addition, we assume that the following separation property is valid for the exact solution, in terms of its gradient:
\begin{equation}
  \partial_X x_e \ge \epsilon_0 ,  \quad \mbox{for $\epsilon_0 > 0$} ,
    \label{assumption:separation}
	\end{equation}
at a point-wise level. The following theorem is the convergence result of the proposed scheme.

\begin{thm} \label{thm:convergence}
Given initial data $x_e (\, \cdot \, ,t=0) \in C^6 (\Omega)$, suppose the exact solution for
the PME equation~\eqref{eqtra}-\eqref{eqtraini} is of regularity class $\mathcal{R}$. Define the numerical error function as $e^n_j = (x_e )^n_j - x^n_j$, at a point-wise level. Then, provided $\dt$ and $h$ are sufficiently small, and under the linear refinement requirement $C_1 h \le \dt \le C_2 h$, we have
	\begin{equation}
\| e^n \|_2 +  \Bigl( \dt   \sum_{m=0}^{n-1} \| \frac12 D_h (e^m + e^{m+1}) \|_2^2 \Bigr)^{1/2}  \le C ( \dt^2 + h^2 ),
	\label{convergence-0}
	\end{equation}
for all positive integers $n$, such that $t_n=n\dt \le T$, where $C>0$ is independent of $n$, $\dt$, and $h$.
	\end{thm}
	
\subsection{Higher order consistency analysis of~\eqref{scheme-2nd-1}:  asymptotic expansion of the numerical solution}
\label{subsec-consistency}

By consistency, the exact solution $x_e$ solves the discrete equation~\eqref{scheme-2nd-1} with second order accuracy in both time and space. Meanwhile, it is observed that this leading local truncation error will not be enough to recover an a-priori $W_h^{1,\infty}$ bound for the temporal derivative of the numerical solution, which is needed in the nonlinear error estimate.
To remedy this, we use a higher order consistency analysis, via a perturbation argument, to recover such a bound in later analysis. In more details, we need to construct supplementary fields, $x_{h,1}$, $x_{\dt,1}$, $x_{\dt, 2}$, and $W$, satisfying
	\begin{equation}
 W = x_e + h^2 x_{h,1} + \dt^2 x_{\dt,1} + \dt^3 x_{\dt, 2} ,
	\label{consistency-1}
	\end{equation}
so that a higher $O (\dt^4 + h^4)$ consistency is satisfied with the given numerical scheme~\eqref{scheme-2nd-1}.  The constructed fields $x_{h,1}$, $x_{\dt,1}$, $x_{\dt, 2}$, which will be obtained using a perturbation expansion, will depend solely on the exact solution $x_e$.

The following truncation error analysis for the temporal discretization can be derived by using a straightforward Taylor expansion
\begin{eqnarray}
   \frac{ (S_e (x_e^n , x_e^{n-1} )  )^{m-1} }{m ( f_0(X) )^{m-2} }
   \cdot \frac{x_e^{n+1}-x_e^{n}}{\dt}  &=&
  - \partial_X \Big[\Big( f_{0}(X) \frac{\ln (\partial_X x_e^{n+1}) - \ln ( \partial_X x_e^n) }{\partial_X ( x_e^{n+1} - x_e^n ) } \Big) \nonumber
\\
  && - A_0 \dt \partial_X ( x_e^{n+1} - x_e^n ) + \dt^2 ( \frac{1}{ \partial_X x_e^{n+1}} - \frac{1}{ \partial_X x_e^n} ) \Big] \nonumber
\\
  &&
  + \dt^2 \g_1^{(0)} + \dt^3 \g_1^{(1)} + O (\dt^4) ,  \label{consistency-2}
\\
  &&
  \mbox{with} \quad
  S_e (x_e^n, x_e^{n-1}) = \partial_X ( \frac32 x_e^n - \frac12 x_e^{n-1} )  .
  \nonumber
\end{eqnarray}
Here the spatial function $\g_j^{(0)}$ is smooth enough in the sense that its derivatives are bounded.

The temporal correction function $x_{\dt,1}$ is given by solving the following equation:
\begin{eqnarray}
  &&
   \frac{ (\partial_X x_e)^{m-1} }{m ( f_0(X) )^{m-2} }
   \partial_t x_{\dt,1}
   +  \frac{ (m-1) (\partial_X x_e)^{m-2} \partial_X x_{\dt,1} }{m ( f_0(X) )^{m-2} }
   \partial_t x_e \nonumber
\\
   &&  \quad
  = - \partial_X \Big( - f_{0}(X) \frac{1}{(\partial_X x_e)^2} \partial_X x_{\dt, 1} \Big)
  - \g_1^{(0)} ,  \label{consistency-3-1}
\\
  &&
  x_{\dt,1}  (0) = x_{\dt,1} (1) =0 ,  \quad
  x_{\dt,1}  (t=0) = 0 .   \label{consistency-3-2}
\end{eqnarray}
Existence of a solution of the above linear PDE is a straightforward, and the solution depends only on the exact solution $x_e$. In addition, the derivatives of $x_{\dt, 1}$ of various orders are bounded. Of course, an application of the semi-implicit discretization to~\eqref{consistency-3-1}-\eqref{consistency-3-2} implies that
\begin{eqnarray}
  &&
   \frac{ (S_e (x_e^n , x_e^{n-1} )  )^{m-1}  }{m ( f_0(X) )^{m-2} }
   \cdot \frac{x_{\dt,1}^{n+1} - x_{\dt, 1}^n}{\dt}
   +  \frac{ (m-1) (S_e (x_e^n , x_e^{n-1} )  )^{m-2}
   S_e (x_{\dt, 1}^n , x_{\dt, 1}^{n-1} )  }{m ( f_0(X) )^{m-2} }
   \cdot \frac{x_e^{n+1} - x_e^n}{\dt}  \nonumber
\\
   &&  \quad
  = - \partial_X \Big( - f_{0}(X) \frac{1}{(\partial_X (\frac12 x_e^{n+1} + x_e^n))^2}
   \cdot \frac12 \partial_X ( x_{\dt, 1}^{n+1} + x_{\dt,1}^n)  \Big)
  - (\g_1^{(0)} )^{n+1/2} + O (\dt^2) ,  \label{consistency-4-1}
\\
  &&
  \mbox{with} \quad
  S_e (x_{\dt, 1}^n, x_{\dt, 1}^{n-1}) = \partial_X ( \frac32 x_{\dt, 1}^n
  - \frac12 x_{\dt, 1}^{n-1} )  .   \label{consistency-4-2}
\end{eqnarray}
Therefore, a combination of~\eqref{consistency-2} and \eqref{consistency-4-1} leads to the third order temporal truncation error for $W_1 = x_e + \dt^2x_{\dt,1}$:
\begin{eqnarray}
   \frac{ (S_e (W_1^n , W_1^{n-1} )  )^{m-1} }{m ( f_0(X) )^{m-2} }
   \cdot \frac{W_1^{n+1} - W_1^{n}}{\dt}  &=&
  - \partial_X \Big[\Big( f_{0}(X) \frac{\ln (\partial_X W_1^{n+1}) - \ln ( \partial_X W_1^n) }{\partial_X ( W_1^{n+1} - W_1^n ) } \Big) \nonumber
\\
  && - A_0 \dt \partial_X ( W_1^{n+1} - W_1^n )
  + \dt^2 ( \frac{1}{ \partial_X W_1^{n+1}} - \frac{1}{ \partial_X W_1^n} ) \Big] \nonumber
\\
  &&
  + \dt^3 \g_1^{(1)} + O (\dt^4) ,  \label{consistency-5-1}
\\
  &&
  \mbox{with} \quad
  S_e (W_1^n, W_1^{n-1}) = \partial_X ( \frac32 W_1^n - \frac12 W_1^{n-1} )  .
  \nonumber
\end{eqnarray}
In the derivation of~\eqref{consistency-5-1}, the following linearized expansions have been utilized:
\begin{eqnarray}
  &&
  \frac{W_1^{n+1} - W_1^n}{\dt} = \frac{x_e^{n+1} - x_e^n}{\dt}  + O (\dt^2) ,
  \label{consistency-6-1}
\\
  &&
   (S_e (x_e^n , x_e^{n-1} )  )^{m-1}
   +   (m-1) (S_e (x_e^n , x_e^{n-1} )  )^{m-2}
   S_e (x_{\dt, 1}^n , x_{\dt, 1}^{n-1} )   \cdot \dt^2 x_{\dt,1}  \nonumber
\\
  &&   \quad
   = S_e (W_1^n, W_1^{n-1})  + O (\dt^4) ,  \label{consistency-6-2}
\\
  &&
  \frac{\ln (\partial_X x_e^{n+1}) - \ln ( \partial_X x_e^n) }{\partial_X ( x_e^{n+1} - x_e^n ) }
  - \dt^2 \frac{1}{(\partial_X (\frac12 x_e^{n+1} + x_e^n))^2}
   \cdot \frac12 \partial_X ( x_{\dt, 1}^{n+1} + x_{\dt,1}^n)    \nonumber
\\
  &&  \quad
  = \frac{\ln (\partial_X W_1^{n+1}) - \ln ( \partial_X W_1^n) }{\partial_X ( W_1^{n+1} - W_1^n ) }  + O (\dt^4) .   \label{consistency-6-3}
\end{eqnarray}

  Similarly, the temporal correction function $x_{\dt,2}$ is given by solving the following equation:
\begin{eqnarray}
  &&
   \frac{ (\partial_X x_e)^{m-1} }{m ( f_0(X) )^{m-2} }
   \partial_t x_{\dt,2}
   +  \frac{ (m-1) (\partial_X x_e)^{m-2} \partial_X x_{\dt,2} }{m ( f_0(X) )^{m-2} }
   \partial_t x_e \nonumber
\\
   &&  \quad
  = - \partial_X \Big( - f_{0}(X) \frac{1}{(\partial_X x_e)^2} \partial_X x_{\dt, 2} \Big)
  - \g_1^{(1)} ,  \label{consistency-7-1}
\\
  &&
  x_{\dt,2}  (0) = x_{\dt,2} (1) =0 ,  \quad
  x_{\dt,2}  (t=0) = 0 ,   \label{consistency-7-2}
\end{eqnarray}
and the solution depends only on the exact solution $x_e$, with derivatives of various orders stay bounded. In turn, an application of the semi-implicit discretization to~\eqref{consistency-7-1}-\eqref{consistency-7-2} implies that
\begin{eqnarray}
  &&
   \frac{ (S_e (x_e^n , x_e^{n-1} )  )^{m-1}  }{m ( f_0(X) )^{m-2} }
   \cdot \frac{x_{\dt,2}^{n+1} - x_{\dt, 2}^n}{\dt}
   +  \frac{ (m-1) (S_e (x_e^n , x_e^{n-1} )  )^{m-2}
   S_e (x_{\dt, 2}^n , x_{\dt, 2}^{n-1} )  }{m ( f_0(X) )^{m-2} }
   \cdot \frac{x_e^{n+1} - x_e^n}{\dt}  \nonumber
\\
   &&  \quad
  = - \partial_X \Big( - f_{0}(X) \frac{1}{(\partial_X (\frac12 x_e^{n+1} + x_e^n))^2}
   \cdot \frac12 \partial_X ( x_{\dt, 2}^{n+1} + x_{\dt,2}^n)  \Big)
  - (\g_1^{(1)} )^{n+1/2} + O (\dt^2) ,  \label{consistency-8-1}
\\
  &&
  \mbox{with} \quad
  S_e (x_{\dt, 2}^n, x_{\dt, 2}^{n-1}) = \partial_X ( \frac32 x_{\dt, 2}^n
  - \frac12 x_{\dt, 2}^{n-1} )  .   \label{consistency-8-2}
\end{eqnarray}
Subsequently, a combination of~\eqref{consistency-5-1} and \eqref{consistency-8-1} yields the fourth order temporal truncation error for $W_2 = W_1 + \dt^3 x_{\dt,2} = x_e + \dt^2 x_{\dt,1} + \dt^3 x_{\dt,2}$:
\begin{eqnarray}
   \frac{ (S_e (W_2^n , W_2^{n-1} )  )^{m-1} }{m ( f_0(X) )^{m-2} }
   \cdot \frac{W_2^{n+1} - W_2^{n}}{\dt}  &=&
  - \partial_X \Big[\Big( f_{0}(X) \frac{\ln (\partial_X W_2^{n+1}) - \ln ( \partial_X W_2^n) }{\partial_X ( W_2^{n+1} - W_2^n ) } \Big) \nonumber
\\
  && - A_0 \dt \partial_X ( W_2^{n+1} - W_2^n )
  + \dt^2 ( \frac{1}{ \partial_X W_2^{n+1}} - \frac{1}{ \partial_X W_2^n} ) \Big] \nonumber
\\
  &&
  + O (\dt^4) ,  \label{consistency-9-1}
\\
  &&
  \mbox{with} \quad
  S_e (W_2^n, W_2^{n-1}) = \partial_X ( \frac32 W_2^n - \frac12 W_2^{n-1} )  ,
  \nonumber
\end{eqnarray}
in which the linearized expansions have been extensively applied.

Next, we construct the spatial correction term $x_{h,1}$ to upgrade the spatial accuracy order. The following truncation error analysis for the spatial discretization can be obtained by using a straightforward Taylor expansion for the constructed profile $W_2$, and exact solution $x_e$:
\begin{eqnarray}
   \frac{ (S_h (W_2^n , W_2^{n-1} )  )^{m-1} }{m ( f_0(X) )^{m-2} }
   \cdot \frac{W_2^{n+1} - W_2^{n}}{\dt}  &=&
  - d_h \Big[\Big( f_{0}(X) \frac{\ln (D_h W_2^{n+1}) - \ln (D_h W_2^n) }{D_h ( W_2^{n+1} - W_2^n ) } \Big) \nonumber
\\
  && - A_0 \dt D_h ( W_2^{n+1} - W_2^n )
  + \dt^2 ( \frac{1}{ D_h W_2^{n+1}} - \frac{1}{ D_h W_2^n} ) \Big] \nonumber
\\
  &&
  + h^2 \h^{(0)} + O (h^4) + O (\dt^4) ,  \label{consistency-10-1}
\\
  &&
  \mbox{with} \quad
  S_h (W_2^n, W_2^{n-1}) = \widetilde{D}_h ( \frac32 W_2^n - \frac12 W_2^{n-1} )  .
  \nonumber
\end{eqnarray}
The spatially discrete function $\h^{(0)}$ is smooth enough in the sense that its discrete derivatives are bounded. We also notice that there is no $O(h^3)$ truncation error term, due to the fact that the  centered difference used in the spatial discretization gives local truncation errors with only even order terms,  $O(h^2)$, $O(h^4)$, etc. Subsequently, the spatial correction function $x_{h, 1}$ is given by solving the following linear PDE:
\begin{eqnarray}
  &&
   \frac{ (\partial_X x_e)^{m-1} }{m ( f_0(X) )^{m-2} }
   \partial_t x_{h,1}
   +  \frac{ (m-1) (\partial_X x_e)^{m-2} \partial_X x_{h,1} }{m ( f_0(X) )^{m-2} }
   \partial_t x_e \nonumber
\\
   &&  \quad
  = - \partial_X \Big( - f_{0}(X) \frac{1}{(\partial_X x_e)^2} \partial_X x_{h,1} \Big)
  - \h^{(0)} ,  \label{consistency-11-1}
\\
  &&
  x_{h,1}  (0) = x_{h,1} (1) =0 ,  \quad
  x_{h,1}  (t=0) = 0 ,   \label{consistency-11-2}
\end{eqnarray}
and the solution depends only on the exact solution $x_e$, with  the  divided differences of various orders stay bounded. In turn, an application of a full discretization to~\eqref{consistency-11-1} implies that
\begin{eqnarray}
  &&
   \frac{ (S_h (x_e^n , x_e^{n-1} )  )^{m-1}  }{m ( f_0(X) )^{m-2} }
   \cdot \frac{x_{h,1}^{n+1} - x_{h,1}^n}{\dt}
   +  \frac{ (m-1) (S_h (x_e^n , x_e^{n-1} )  )^{m-2}
   S_h (x_{h,1}^n , x_{h,1}^{n-1} )  }{m ( f_0(X) )^{m-2} }
   \cdot \frac{x_e^{n+1} - x_e^n}{\dt}  \nonumber
\\
   &&  \quad
  = - d_h \Big( - f_{0}(X) \frac{1}{(D_h (\frac12 x_e^{n+1} + x_e^n))^2}
   \cdot \frac12 D_h ( x_{h,1}^{n+1} + x_{h,1}^n)  \Big)
  - (\h^{(0)} )^{n+1/2} + O (h^2) ,  \label{consistency-12-1}
\\
  &&
  \mbox{with} \quad
  S_h (x_{h,1}^n, x_{h,1}^{n-1}) = \widetilde{D}_h ( \frac32 x_{h,1}^n
  - \frac12 x_{h,1}^{n-1} )  .   \label{consistency-12-2}
\end{eqnarray}
Finally, a combination of~\eqref{consistency-9-1} and \eqref{consistency-12-1} yields the fourth order temporal truncation error for $W = W_2 + h^2 x_{h,1} = x_e + h^2 x_{h,1} + \dt^2 x_{\dt,1} + \dt^3 x_{\dt,2}$ (as given by~\eqref{consistency-1}):
\begin{eqnarray}
   \frac{ (S_h (W^n , W^{n-1} )  )^{m-1} }{m ( f_0(X) )^{m-2} }
   \cdot \frac{W^{n+1} - W^{n}}{\dt}  &=&
  - d_h \Big[\Big( f_{0}(X) \frac{\ln (D_h W^{n+1}) - \ln ( D_h W^n) }{D_h ( W^{n+1} - W^n ) } \Big) \nonumber
\\
  && - A_0 \dt D_h ( W^{n+1} - W^n )
  + \dt^2 ( \frac{1}{ D_h W^{n+1}} - \frac{1}{D_h W^n} ) \Big] \nonumber
\\
  &&
  + \tau^n ,  \quad \mbox{with $\| \tau^n \|_2 \le C (\dt^4 + h^4)$},
  \label{consistency-13-1}
\\
  &&
  \mbox{and} \, \, \,
  S_h (W^n, W^{n-1}) = \widetilde{D}_h ( \frac32 W^n - \frac12 W^{n-1} )  .
  \nonumber
\end{eqnarray}
Again, the linearized expansions have been extensively applied.

\begin{rem}
Since the temporal and spatial correction functions, namely $x_{\dt, 1}$, $x_{\dt, 2}$ and $x_{h,1}$, are bounded, we recall the separation property~\eqref{assumption:separation} for the exact solution, and obtain a similar property for the constructed profile $W$:
\begin{equation}
  D_h W  \ge \epsilon_0^* , \quad \mbox{for $\epsilon_0^* > 0$} .
    \label{assumption:separation-2}
	\end{equation}
Such a uniform bound will be used in the convergence analysis.

For the the constructed profile $W$, we also assume its discrete $W^{2,\infty}$ bound, as well as the $W^{1,\infty}$ bound for its temporal derivatives:
\begin{eqnarray}
  \hspace {-0.4in} &&
  \| D_h W \|_\infty + \| D_h^2 W \|_\infty \le C^* ,  \quad \
  \nrm{ D_t W^n }_\infty + \nrm{ D_h D_t W^n}_\infty
  + \nrm{ D_h ( D_t^2 W^n) }_\infty \le C^* ,
    \label{assumption:constructed bound}
\\
  \hspace{-0.4in} &&
  \mbox{with} \quad D_t W^n := \frac{W^{n+1} - W^n}{\dt} ,  \, \, \,
  D_t^2 W^n = \frac{W^{n+1} - 2 W^n + W^{n-1}}{\dt^2}  ,   \nonumber
	\end{eqnarray}
which comes from the regularity of the exact solution $x_e$ and the correction functions.
\end{rem}

\begin{rem}
The aim for such a higher order asymptotic expansion and truncation error estimate is to justify an a-priori $W_h^{1,\infty}$ bound of the numerical solution, which is needed to obtain the phase separation property, similarly formulated as~\eqref{assumption:separation-2} for the constructed approximate solution. In addition, a discrete $W_h^{1,\infty}$ bound for the temporal derivatives of the numerical solution is also needed in the nonlinear analysis, which turns out to be the key reason to derive a fourth order consistency estimate for the constructed solution.
\end{rem}

\subsection{A preliminary rough error estimate}

Instead of a direct analysis for the error function defined as $e^m = x_e^m - x^m$, we introduce an alternate numerical error function:
	\begin{equation}
\tilde{x}^m := W^m - x^m .
	\label{error function-2}
	\end{equation}
The advantage of such a numerical error function is associated with its higher order accuracy, which comes from the higher order consistency estimate~\eqref{consistency-13-1}. Moreover, the following notations are introduced, for the convenience of the analysis presented later.
\begin{eqnarray}
  &&
  x^{n+\nicefrac12} = \frac12 (x^{n+1} + x^n) ,  \, \, \,
  W^{n+\nicefrac12} = \frac12 (W^{n+1} + W^n)  , \label{notation-1}
\\
  &&
  \breve{x}^{n+\nicefrac12} = \frac32 x^n - \frac12 x^{n-1} ,  \, \, \,
  \breve{W}^{n+\nicefrac12} = \frac32 W^n - \frac12 W^{n-1} ,  \label{notation-2}
\\
  &&
  \tilde{x}^{n+\nicefrac12} =  W^{n+\nicefrac12} - x^{n+\nicefrac12}
  = \frac12 (\tilde{x}^{n+1} + \tilde{x}^n) ,  \, \, \,
 \breve{\tilde{x}}^{n+\nicefrac12}
 = \breve{W}^{n+\nicefrac12} - \breve{x}^{n+\nicefrac12}
 = \frac32 \tilde{x}^n - \frac12 \tilde{x}^{n-1} .  \label{error function-2}
\end{eqnarray}

In turn, subtracting the numerical scheme~\eqref{scheme-2nd-1} from the consistency estimate~\eqref{consistency-13-1} yields
\begin{eqnarray}
  &&
   \frac{ (S_h (x^n , x^{n-1} )  )^{m-1} }{m ( f_0(X) )^{m-2} }
   \cdot \frac{\tilde{x}^{n+1} - \tilde{x}^{n}}{\dt}
   + \frac{ (S_h (W^n , W^{n-1} )  )^{m-1} - (S_h (x^n , x^{n-1} )  )^{m-1} }{m ( f_0(X) )^{m-2} }  \cdot \frac{W^{n+1} - W^{n}}{\dt}    \nonumber
\\
   &=&
  - d_h \Big[ f_{0}(X) \Big(
    \frac{\ln (D_h W^{n+1}) - \ln ( D_h W^n) }{D_h ( W^{n+1} - W^n ) }
  - \frac{\ln (D_h x^{n+1}) - \ln ( D_h x^n) }{D_h ( x^{n+1} - x^n ) }  \Big) \nonumber
\\
  && - A_0 \dt D_h ( \tilde{x}^{n+1} - \tilde{x}^n )
  - \dt^2 ( \frac{D_h \tilde{x}^{n+1} }{ D_h W^{n+1} D_h x^{n+1}}
  - \frac{D_h \tilde{x}^n}{D_h W^n D_h x^n} )  \Big] \nonumber
\\
  &&
  + \tau^n ,  \quad \mbox{with $\| \tau^n \|_2 \le C (\dt^4 + h^4)$} .
  \label{error equation-1}
\end{eqnarray}

To proceed with the nonlinear analysis, we make the following a-priori assumption at the previous time steps, for $k=n, n-1, n-2$:
\begin{eqnarray}
  \| \tilde{x}^k \|_2 \le {\cal C} ( \dt^4 + h^4)  ,   \quad
  \mbox{with ${\cal C}$ uniform for a fixed final time $T$} .
   \label{a priori-1}
\end{eqnarray}
Such an a-priori assumption will be recovered by the optimal rate convergence analysis at the next time step, as will be demonstrated later. With this assumption, the following bounds for the numerical solution are obtained, with the help of inverse inequality:
\begin{eqnarray}
  &&
  \| D_h \tilde{x}^k \|_2 \le \frac{C \| \tilde{x}^k \|_2}{h}
  \le C {\cal C} (\dt^3 + h^3) ,   \label{prelim bound-1}
\\
  &&
  \| D_h \tilde{x}^k \|_\infty \le \frac{C \| \tilde{x}^k \|_2}{h^\frac32}
  \le C {\cal C} (\dt^\frac52 + h^\frac52) \le \frac{\epsilon_0^*}{2} ,   \label{prelim bound-2}
\\
  &&
  \mbox{so that} \, \,
   \frac{\epsilon_0^*}{2} \le D_h x^k = D_h W^k - D_h \tilde{x}^k
   \le C^* + \frac{\epsilon_0^*}{2} := \tilde{C}^* ,
  \label{prelim bound-3}
\end{eqnarray}
for $k=n, n-1, n-2$, in which the lower and upper bounds~\eqref{assumption:separation-2}, \eqref{assumption:constructed bound}, for the constructed profile $W$, have been used. In addition, the following observation is made, motived by the preliminary estimate~\eqref{prelim bound-2}:
\begin{eqnarray}
  \| D_h x^k - D_h x^{k-1} \|_\infty
  &\le& \| D_h W^k - D_h W^{k-1} \|_\infty
  +  \| D_h \tilde{x}^k - D_h \tilde{x}^{k-1} \|_\infty  \nonumber
\\
  &\le&
  C^* \dt + C {\cal C} (\dt^\frac52 + h^\frac52)
  \le (C^*+ 1) \dt ,   \label{prelim bound-4}
\end{eqnarray}
for $k=n, n-1$, in which the regularity requirement~\eqref{assumption:constructed bound} for the constructed solution has been applied again. In turn, we conclude that
\begin{eqnarray}
  &&
  \widetilde{D}_h (\frac32 x^k - \frac12 x^{k-1} )
  = \widetilde{D}_h x^k + \frac12 \widetilde{D}_h ( x^k - x^{k-1} )
  \ge \frac{\epsilon_0^*}{2} - \frac{\epsilon_0^*}{4}
  = \frac{\epsilon_0^*}{4}   \ge \dt^2 ,  \label{prelim bound-5}
\\
  &&
  \mbox{so that} \, \, \, S_h (x^k, x^{k-1}) = \widetilde{D}_h \breve{x}^{n+\nicefrac12}
  = \widetilde{D}_h (\frac32 x^k - \frac12 x^{k-1} )  ,  \, \, \,
  \mbox{for $k=n, n-1$} .  \label{prelim bound-5-b}
\end{eqnarray}

The following preliminary estimates are needed in the nonlinear error analysis for the left hand side two terms of~\eqref{error equation-1}.

\begin{lem}   \label{lem:prelim est}
For the constructed profile $W$ satisfying~\eqref{assumption:constructed bound}, and the numerical error function with a discrete $W^{1,\infty}$ bound given by~\eqref{prelim bound-2}, for $k=n, n-1, n-2$, we have
\begin{eqnarray}
  \hspace{-0.4in}  &&
  \| (S_h (x^n , x^{n-1} )  )^{m-1}  -  (S_h (x^{n-1} , x^{n-2} )  )^{m-1} \|_\infty
  \le  \tilde{C}_1 \dt ,
  \label{prelim est-1}
\\
  \hspace{-0.4in}  &&
  (S_h (W^n , W^{n-1} )  )^{m-1} - (S_h (x^n , x^{n-1} )  )^{m-1}
  = {\cal N}^{n+\nicefrac12}  \widetilde{D}_h \breve{\tilde{x}}^{n+\nicefrac12} ,
  \quad \mbox{with} \, \, \,
  \| {\cal N}^{n+\nicefrac12} \|_\infty \le \tilde{C}_2 ,   \label{prelim est-2}
\\
  \hspace{-0.4in}  &&   \mbox{and} \, \, \,
  \| D_h  {\cal N}^{n+\nicefrac12} \|_\infty \le \tilde{C}_3 ,   \, \, \,
  \mbox{if we define} \, \,
  ({\cal N}^{n+\nicefrac12} )_0  = ({\cal N}^{n+\nicefrac12} )_1 , \,
  ({\cal N}^{n+\nicefrac12} )_M  = ({\cal N}^{n+\nicefrac12} )_{M-1} , 
  \label{prelim est-3}
 \end{eqnarray}
in which $\tilde{C}_1$, $\tilde{C}_2$ and $\tilde{C}_3$ is only dependent on the exact solution, independent on $\dt$ and $h$.
\end{lem}

\begin{proof}
Based on the representation identity~\eqref{prelim bound-5-b}, we apply the intermediate value theorem and see that
\begin{eqnarray}
  &&
  (S_h (x^n , x^{n-1} )  )^{m-1}  -  (S_h (x^{n-1} , x^{n-2} )  )^{m-1}    \nonumber
\\
  &=&
  (m-1) ( \xi^{(1)} )^{m-2}
  \widetilde{D}_h \Big(  \frac32 (x^n - x^{n-1} )  - \frac12 (x^{n-1} - x^{n-2} )   \Bigr)  ,
  \label{prelim est-4}
\end{eqnarray}
with $\xi^{(1)}$ between $\widetilde{D}_h \Big(  \frac32 x^n - \frac12 x^{n-1} \Big)$ and
$\widetilde{D}_h \Big(  \frac32 x^{n-1} - \frac12 x^{n-2} \Big)$. Meanwhile, by the upper estimate~\eqref{prelim bound-3} and the lower bound~\eqref{prelim bound-5}, we get
\begin{eqnarray}
  \frac{\epsilon_0^*}{4}  \le  \widetilde{D}_h (\frac32 x^k - \frac12 x^{k-1} )
  \le \frac{3 \tilde{C}^*}{2}  ,  \, \, \, \mbox{for $k=n, n-1$, so that}  \,  \, \,
  \frac{\epsilon_0^*}{4}  \le  \xi^{(1)}  \le \frac{3 \tilde{C}^*}{2}  .
  \label{prelim est-5}
\end{eqnarray}
A substitution into~\eqref{prelim est-4}, combined with the estimate~\eqref{prelim bound-4}, indicates the desired inequality, with $\tilde{C}_1 = 2 (m-1) \max ( \frac{4}{\epsilon_0^*}, \frac{3 \tilde{C}^*}{2} )^{|m-2|} (C^*+1)$.

A similar application of intermediate value theorem reveals that
\begin{eqnarray}
  (S_h (W^n , W^{n-1} )  )^{m-1} - (S_h (x^n , x^{n-1} )  )^{m-1}
   =  (m-1) ( \xi^{(2)} )^{m-2} \widetilde{D}_h \breve{\tilde{x}}^{n+ \nicefrac12} ,
      \label{prelim est-6}
\end{eqnarray}
and $\xi^{(2)}$ between $\widetilde{D}_h \Big(  \frac32 W^n - \frac12 W^{n-1} \Big)$ and
$\widetilde{D}_h \Big(  \frac32 x^n - \frac12 x^{n-1} \Big)$. Using the same argument as in~\eqref{prelim est-5}, we see that $\frac{\epsilon_0^*}{4}  \le  \xi^{(2)}  \le \frac{3 \tilde{C}^*}{2}$, so that
\begin{eqnarray}
  \| {\cal N}^{n+\nicefrac12} \|_\infty  = \| (m-1) ( \xi^{(2)} )^{m-2}  \|_\infty
  \le \tilde{C}_2 :=(m-1) \max ( \frac{4}{\epsilon_0^*}, \frac{3 \tilde{C}^*}{2} )^{|m-2|} ,
    \label{prelim est-7}
\end{eqnarray}
which completes the proof of~\eqref{prelim est-2}. Moreover, for two adjacent grid points $x_i$ and $x_{i+1}$ (with $1 \le i \le M-2$), motivated by the fact that
\begin{eqnarray}
  \xi^{(2)}_i \, \mbox{is between} \, \widetilde{D}_h \breve{W}^{n+\nicefrac12}_i \, \mbox{and} \, \widetilde{D}_h \breve{x}^{n+\nicefrac12}_i ,  \, \, \,
  \xi^{(2)}_{i+1} \, \mbox{is between} \, \widetilde{D}_h \breve{W}^{n+\nicefrac12}_{i+1} \, \mbox{and} \, \widetilde{D}_h \breve{x}^{n+\nicefrac12}_{i+1} ,
  \label{prelim est-8-1}
\end{eqnarray}
we have the following observation:
\begin{eqnarray}
  | \xi^{(2)}_{i+1} - \xi^{(2)}_i |  &\le&
  \Big| \widetilde{D}_h \breve{W}^{n+\nicefrac12}_{i+1}
  - \widetilde{D}_h \breve{W}^{n+\nicefrac12}_i  \Big|
  + \Big| \widetilde{D}_h  \breve{\tilde{x}}^{n+\nicefrac12}_i \Big|
  + \Big| \widetilde{D}_h  \breve{\tilde{x}}^{n+\nicefrac12}_{i+1} \Big|  \nonumber
\\
  &\le&
  \frac12 h \Big( \Big| D_h^2 \breve{W}^{n+\nicefrac12}_{i+1} \Big|
     + \Big| D_h^2 \breve{W}^{n+\nicefrac12}_i \Big| \Bigr)
     + C {\cal C} (\dt^\frac52 + h^\frac52)   \nonumber
\\
  &\le&
  C^* h + h = (C^* +1) h ,  \label{prelim est-8-2}
\end{eqnarray}
in which the preliminary estimate~\eqref{prelim bound-2} and the regularity assumption~\eqref{assumption:constructed bound} for the constructed profile have been used. On the other hand, by another application of intermediate value theorem:
\begin{eqnarray}
  {\cal N}^{n+\nicefrac12}_{i+1} - {\cal N}^{n+\nicefrac12}_i
  &=&  (m-1) ( \xi^{(2)}_{i+1} )^{m-2}  -  (m-1) ( \xi^{(2)}_i )^{m-2}  \nonumber
\\
  &=&
  (m-1) (m-2) ( \eta^{(1)} )^{m-3}  ( \xi^{(2)}_{i+1} - \xi^{(2)}_i  ) ,
    \label{prelim est-9-1}
\end{eqnarray}
with $\eta^{(1)}$ between $\xi^{(2)}_i$ and $\xi^{(2)}_{i+1}$, we get the desired estimate
\begin{eqnarray}
  \Big| {\cal N}^{n+\nicefrac12}_{i+1} - {\cal N}^{n+\nicefrac12}_i  \Big|
  \le (m-1) |m-2| \max ( \frac{4}{\epsilon_0^*}, \frac{3 \tilde{C}^*}{2} )^{|m-3|}
  (C^* +1 )  h .  \label{prelim est-9-2}
\end{eqnarray}
This completes the proof of~\eqref{prelim est-3}, by taking $\tilde{C}_3 = (m-1) |m-2| \max ( \frac{4}{\epsilon_0^*}, \frac{3 \tilde{C}^*}{2} )^{|m-3|} (C^* +1 )$.
\end{proof}

Now we proceed with a rough error estimate. Taking a discrete inner product with~\eqref{error equation-1} by $2 \tilde{x}^{n+1}$ leads to
\begin{eqnarray}
  &&
   \Bigl\langle \frac{ (S_h (x^n , x^{n-1} )  )^{m-1} }{m ( f_0(X) )^{m-2} }
   \cdot \frac{\tilde{x}^{n+1} - \tilde{x}^n}{\dt} , 2 \tilde{x}^{n+1} \Bigr\rangle
   +  2 A_0 \dt \langle D_h ( \tilde{x}^{n+1} - \tilde{x}^n ) ,
   D_h \tilde{x}^{n+1}  \rangle
   \nonumber
\\
  &&
  - 2 \Bigl\langle  f_{0}(X) \Big(
    \frac{\ln (D_h W^{n+1}) - \ln ( D_h W^n) }{D_h ( W^{n+1} - W^n ) }
  - \frac{\ln (D_h x^{n+1}) - \ln ( D_h x^n) }{D_h ( x^{n+1} - x^n ) }  \Big) ,
   \tilde{x}^{n+1}  \Bigr\rangle \nonumber
\\
  &=&
   - 2 \Bigl\langle \frac{ (S_h (W^n , W^{n-1} )  )^{m-1} - (S_h (x^n , x^{n-1} )  )^{m-1} }{m ( f_0(X) )^{m-2} }  \cdot \frac{W^{n+1} - W^{n}}{\dt}  ,
     \tilde{x}^{n+1} \Bigr\rangle  \nonumber
\\
  &&
  - 2 \dt^2 \Bigl\langle \frac{D_h \tilde{x}^{n+1} }{ D_h W^{n+1} D_h x^{n+1}}
  - \frac{D_h \tilde{x}^n}{D_h W^n D_h x^n} ,
   D_h \tilde{x}^{n+1} \Bigr\rangle
  + 2 \langle \tau^n ,  \tilde{x}^{n+1}  \rangle .
  \label{convergence-rough-1}
\end{eqnarray}

For the temporal derivative term, we make use of the equality
\begin{eqnarray}
   2 \tilde{x}^{n+1}  (\tilde{x}^{n+1} - \tilde{x}^n)
   = (\tilde{x}^{n+1})^2 - (\tilde{x}^n)^2  + (\tilde{x}^{n+1} - \tilde{x}^n)^2
   \ge (\tilde{x}^{n+1})^2 - (\tilde{x}^n)^2 ,   \label{convergence-rough-2-1}
\end{eqnarray}
and get
\begin{eqnarray}
  &&
  \Bigl\langle \frac{ (S_h (x^n , x^{n-1} )  )^{m-1} }{m ( f_0(X) )^{m-2} }
   \cdot \frac{\tilde{x}^{n+1} - \tilde{x}^n}{\dt} ,
   2 \tilde{x}^{n+1} \Bigr\rangle   \nonumber
 \\
   &\ge&
   \frac{1}{\dt}  \Bigl(
   \Bigl\langle \frac{ (S_h (x^n , x^{n-1} )  )^{m-1} }{m ( f_0(X) )^{m-2} }  ,
    ( \tilde{x}^{n+1} )^2 \Bigr\rangle
   - \Bigl\langle \frac{ (S_h (x^n , x^{n-1} )  )^{m-1} }{m ( f_0(X) )^{m-2} }  ,
    ( \tilde{x}^n )^2 \Bigr\rangle  \Bigr)   \nonumber
\\
  &=&
  \frac{1}{\dt}  \Bigl(
   \Bigl\langle \frac{ (S_h (x^n , x^{n-1} )  )^{m-1} }{m ( f_0(X) )^{m-2} }  ,
    ( \tilde{x}^{n+1} )^2 \Bigr\rangle
   - \Bigl\langle \frac{ (S_h (x^{n-1} , x^{n-2} )  )^{m-1} }{m ( f_0(X) )^{m-2} }  ,
    ( \tilde{x}^n )^2 \Bigr\rangle  \Bigr)   \nonumber
\\
  &&
  - \frac{1}{\dt}
   \Bigl\langle \frac{ (S_h (x^n , x^{n-1} )  )^{m-1} - (S_h (x^{n-1} , x^{n-2} )  )^{m-1} }{m ( f_0(X) )^{m-2} }  ,   ( \tilde{x}^n )^2 \Bigr\rangle  .    \label{convergence-rough-2-2}
\end{eqnarray}
Furthermore, with an application of the preliminary estimate~\eqref{prelim est-1}, the last term of~\eqref{convergence-rough-2-2} could be bounded as
\begin{eqnarray}
  &&
     \frac{1}{\dt}
   \Bigl\langle \frac{ (S_h (x^n , x^{n-1} )  )^{m-1} - (S_h (x^{n-1} , x^{n-2} )  )^{m-1} }{m ( f_0(X) )^{m-2} }  ,   ( \tilde{x}^n )^2 \Bigr\rangle
   \le \tilde{C}_4 \| \tilde{x}^n \|_2^2 ,  \label{convergence-rough-2-3}
\\
  &&  \mbox{with} \, \, \,
  \tilde{C}_4 =  \tilde{C}_1 \max \Big( \frac{1}{\min_\Omega f_0 (X)} , \max_\Omega f_0 (X) \Big)^{|m-2|}  .   \nonumber
\end{eqnarray}

The second term on the left hand side of~\eqref{convergence-rough-1} could be controlled by a standard inequality:
\begin{eqnarray}
   2 A_0 \dt \langle D_h ( \tilde{x}^{n+1} - \tilde{x}^n ) ,
   D_h \tilde{x}^{n+1}  \rangle
   \ge  A_0 \dt ( \| D_h  \tilde{x}^{n+1} \|_2^2  - \| D_h \tilde{x}^n \|_2 ) .
     \label{convergence-rough-3}
\end{eqnarray}
The Cauchy inequality could be applied to bound the term associated with the local truncation error:
\begin{eqnarray}
   2 \langle \tau^n ,  \tilde{x}^{n+1}  \rangle
   \le \| \tau^n \|_2^2 +  \| \tilde{x}^{n+1} \|_2^2 .  \label{convergence-rough-4}
\end{eqnarray}

For the second term on the right hand side of~\eqref{convergence-rough-1}, we make the following observations:
\begin{eqnarray}
  - 2 \dt^2 \Bigl\langle \frac{D_h \tilde{x}^{n+1} }{ D_h W^{n+1} D_h x^{n+1}}  ,
     D_h \tilde{x}^{n+1} \Bigr\rangle  &\le& 0 ,   \label{convergence-rough-5-1}
\\
   2 \dt^2 \Bigl\langle \frac{D_h \tilde{x}^n}{D_h W^n D_h x^n}  ,
     D_h \tilde{x}^{n+1} \Bigr\rangle
     &\le& 2 \dt^2 \cdot \frac{1}{\frac12 (\epsilon_0^*)^2}
     \| D_h \tilde{x}^n \|_2 \cdot \| D_h \tilde{x}^{n+1} \|_2   \nonumber
\\
  &\le&
  \dt^3 \| D_h \tilde{x}^{n+1} \|_2^2
  + 4 \dt (\epsilon_0^*)^{-4}  \| D_h \tilde{x}^n \|_2^2    \nonumber
\\
  &\le&
  \tilde{C}_5 \dt \| \tilde{x}^{n+1} \|_2^2
  + 4 \dt (\epsilon_0^*)^{-4}  \| D_h \tilde{x}^n \|_2^2 ,
   \label{convergence-rough-5-2}
\end{eqnarray}
in which~\eqref{convergence-rough-5-1} is based on the fact that $D_h W^{n+1} >0$, $D_h x^{n+1} >0$ (as given by the unique solvability analysis in Theorem~\ref{thm:solvability}), the first step of~\eqref{convergence-rough-5-2} comes from the separation property~\eqref{assumption:separation-2} (for the constructed profile $W$) and the preliminary estimate~\eqref{prelim bound-3}, and an inverse inequality has been applied at the last step.

For the first term on the right hand side of~\eqref{convergence-rough-1}, the preliminary estimate~\eqref{prelim est-2} and the regularity assumption~\eqref{assumption:constructed bound} have to be applied:
\begin{eqnarray}
  &&
   - 2 \Bigl\langle \frac{ (S_h (W^n , W^{n-1} )  )^{m-1} - (S_h (x^n , x^{n-1} )  )^{m-1} }{m ( f_0(X) )^{m-2} }  \cdot \frac{W^{n+1} - W^{n}}{\dt}  ,
     \tilde{x}^{n+1} \Bigr\rangle   \nonumber
\\
  &=&
  - 2 \Bigl\langle \frac{ {\cal N}^{n+\nicefrac12}  \widetilde{D}_h \breve{\tilde{x}}^{n+\nicefrac12} }{m ( f_0(X) )^{m-2} }  \cdot \frac{W^{n+1} - W^{n}}{\dt}  ,
     \tilde{x}^{n+1} \Bigr\rangle    \nonumber
\\
  &\le&
     \tilde{C}_6     \| \widetilde{D}_h \breve{\tilde{x}}^{n+\nicefrac12}  \|_2
     \cdot \| \tilde{x}^{n+1} \|_2
     \le \frac{\tilde{C}_6}{2} (
     \| D_h \breve{\tilde{x}}^{n+\nicefrac12}  \|_2^2
     + \| \tilde{x}^{n+1} \|_2^2 ) ,
   \label{convergence-rough-6-1}
\\
  &&   \mbox{with} \, \, \,
  \tilde{C}_6 = \frac{2}{m}  \tilde{C}_2 C^* \cdot \max \Big( \frac{1}{\min_\Omega f_0 (X)} , \max_\Omega f_0 (X) \Big)^{|m-2|}  .  \nonumber
\end{eqnarray}
Notice that the inequality $\| \widetilde{D}_h f \|_2 \le \| D_h f \|_2$ has been used in the last step.

The rest analysis is focused on the term associated with the nonlinear diffusion part:  \begin{eqnarray}
  {\cal NLE}^{n+\nicefrac12} :=
  - \frac{\ln (D_h W^{n+1}) - \ln ( D_h W^n) }{D_h ( W^{n+1} - W^n ) }
  + \frac{\ln (D_h x^{n+1}) - \ln ( D_h x^n) }{D_h ( x^{n+1} - x^n ) }  .
  \label{notation-3}
\end{eqnarray}
The following lemma is needed in the nonlinear estimate; its proof will be provided in the Appendix~\ref{sec:proof lemma 2}.

\begin{lem} \label{lem:convexity-1}
Fix $x_0 > 0$, and we define $q_1 (x) := - \frac{\ln x - \ln x_0}{x - x_0}$ for $x >0$. The following properties are valid:
\begin{eqnarray}
  &&
  q'_1 (x) > 0 , \, \, \mbox{for any $x>0$} ,    \label{q convexity-1}
\\
  &&
  q_1 (y) - q_1 (x) = q'_1 (\eta) (y-x) ,  \, \, \mbox{with $q'_1 (\eta)$ between $\frac{1}{2 y^2}$, $\frac{1}{2 x^2}$ and $\frac{1}{2 x_0^2}$,  $\forall x>0, \, y>0$} ,
  \label{q convexity-2}
\\
  &&
  q''_1 (x) \le 0 , \, \, \mbox{for any $x>0$} ,    \label{q convexity-3}
\\
  &&
  \frac{q_1 (y) - q_1 (x)}{y -x} \, \, \mbox{is a decreasing function of $x$, for any fixed $y>0$} .  \label{q convexity-4}
\end{eqnarray}
\end{lem}

Subsequently, the following point-wise nonlinear estimate becomes available for ${\cal NLE}^{n+\nicefrac12}$.

\begin{lem} \label{lem:nonlinear-1}
At each numerical mesh cell $(x_i, x_{i+1})$, we have the following estimate
\begin{eqnarray}
  &&
  {\cal NLE}^{n+\nicefrac12}_{i+\nicefrac12}
  = \xi^{(3)}_{i+\nicefrac12} D_h \tilde{x}^{n+1}_{i+\nicefrac12}
  + \xi^{(4)}_{i+\nicefrac12} D_h \tilde{x}^n_{i+\nicefrac12}  ,
  \label{nonlinear est-0-1}
\\
  &&   \mbox{with} \, \, \,
  \xi^{(3)}_{i+\nicefrac12} \ge \frac{1}{2 \tilde{C}^*} h ,  \, \,
  \frac{1}{2 (\tilde{C}^*)^2} \le \xi^{(4)}_{i+\nicefrac12} \le 2 ( \epsilon_0^*)^{-2} .
  \label{nonlinear est-0-2}
\end{eqnarray}
\end{lem}

\begin{proof}
The following decomposition is performed for ${\cal NLE}$:
\begin{eqnarray}
  {\cal NLE}^{n+\nicefrac12}
  &=&  {\cal NLE}^{n+\nicefrac12,(1)}  + {\cal NLE}^{n+\nicefrac12,(2)}, \, \, \,
  \mbox{with}   \nonumber
\\
  {\cal NLE}^{n+\nicefrac12,(1)}
  &=& - \frac{\ln (D_h W^{n+1}) - \ln ( D_h x^n) }{D_h W^{n+1} - D_h x^n }
  + \frac{\ln (D_h x^{n+1}) - \ln ( D_h x^n) }{D_h x^{n+1} - D_h x^n }   \nonumber
\\
  &=&
    q_1 (D_h W^{n+1}) - q_1 (D_h x^{n+1}) ,   \, \, \,
  \mbox{with fixed $x_0 = D_h x^n$} ,   \label{nonlinear est-1-1}
\\
  {\cal NLE}^{n+\nicefrac12,(2)}
  &=& - \frac{\ln (D_h W^{n+1}) - \ln ( D_h W^n) }{D_h W^{n+1} - D_h W^n }
  + \frac{\ln (D_h W^{n+1}) - \ln ( D_h x^n) }{D_h W^{n+1} - D_h x^n }   \nonumber
\\
  &=&
    q_1 (D_h W^n) - q_1 (D_h x^n) ,   \, \, \,
  \mbox{with fixed $x_0 = D_h W^{n+1}$} .   \label{nonlinear est-1-2}
\end{eqnarray}
For the first part ${\cal NLE}^{n+\nicefrac12,(1)}$, we make use of the following bound for $D_h x^{n+1}$:
\begin{eqnarray}
  0 < (D_h x^{n+1} )_{i+ \nicefrac12} \le \frac{1}{h} ,   \quad
  \mbox{since} \, \, 0 \le  x_k^{n+1} \le 1 ,   \, \, \forall  0 \le k \le M ,
   \label{nonlinear est-2}
\end{eqnarray}
so that an application of property~\eqref{q convexity-4} implies that
\begin{eqnarray}
   \frac{q_1 (D_h W^{n+1}) - q_1 (D_h x^{n+1}) }{ D_h W^{n+1} - D_h x^{n+1} }
   &\ge&  \frac{q_1 (D_h W^{n+1}) - q_1 (\frac{1}{h}) }{ D_h W^{n+1} - \frac{1}{h} }
   \nonumber
\\
  &=&
  \frac{-\frac{\ln \frac{1}{h} - \ln ( D_h x^n) }{\frac{1}{h} - D_h x^n }
  - q_1 (D_h W^{n+1}) }{ \frac{1}{h} - D_h W^{n+1} }
  \ge  \frac{1}{2 \tilde{C}^*} h ,  \label{nonlinear est-3-1}
\end{eqnarray}
in which the last step is based on the preliminary estimate~\eqref{prelim bound-3}, as well as the fact that, the value of $-q_1 (D_h W^{n+1})$ is between $\frac{1}{D_h W^{n+1}}$ and $\frac{1}{D_h x^n}$. This inequality is equivalent to
\begin{eqnarray}
   {\cal NLE}^{n+\nicefrac12,(1)}_{i+\nicefrac12}
   = \xi^{(3)}_{i+\nicefrac12} D_h \tilde{x}^{n+1}_{i+\nicefrac12} ,   \quad
   \mbox{with} \, \, \, \xi^{(3)}_{i+\nicefrac12} \ge \frac{1}{2 \tilde{C}^*} h .
    \label{nonlinear est-3-2}
\end{eqnarray}
For the first part ${\cal NLE}^{n+\nicefrac12,(2)}$, we apply property~\eqref{q convexity-2} so that
\begin{eqnarray}
  &&
   {\cal NLE}^{n+\nicefrac12,(2)}
  = q_1 (D_h W^n) - q_1 (D_h x^n)
  = q'_1 (\eta) D_h \tilde{x}^n ,  \label{nonlinear est-4-1}
\\
  &&  \mbox{with} \, \, \, q'_1 (\eta)  \, \mbox{between} \, \,
  \frac{1}{2 (D_h W^n)^2} , \, \, \frac{1}{2 (D_h x^n)^2} , \, \,  \mbox{and} \, \,
   \frac{1}{2 (D_h W^{n+1})^2} .
    \label{nonlinear est-4-2}
\end{eqnarray}
On the other hand, by the separation property~\eqref{assumption:separation-2}, the regularity assumption~\eqref{assumption:constructed bound} for $W$, combined with the preliminary estimate~\eqref{prelim bound-3}, we obtain the desired estimate:
\begin{eqnarray}
  \frac{1}{2 (\tilde{C}^*)^2} \le q'_1 (\eta) \le 2 ( \epsilon_0^*)^{-2} .
  \label{nonlinear est-4-3}
\end{eqnarray}
In other words, the second estimate in~\eqref{nonlinear est-0-1} becomes available:
\begin{eqnarray}
  {\cal NLE}^{n+\nicefrac12,(2)}_{i+\nicefrac12}
   = \xi^{(4)}_{i+\nicefrac12} D_h \tilde{x}^n_{i+\nicefrac12}  ,  \quad
   \mbox{with} \, \, \,
  \frac{1}{2 (\tilde{C}^*)^2} \le \xi^{(4)}_{i+\nicefrac12} \le 2 ( \epsilon_0^*)^{-2} .
  \label{nonlinear est-4-4}
\end{eqnarray}
This completes the proof of Lemma~\ref{lem:nonlinear-1}.
\end{proof}

As a consequence of this lemma, we analyze the nonlinear product at each cell $(x_i, x_{i+1})$:
\begin{eqnarray}
  &&
  f_0 (X_{i+\nicefrac12}) {\cal NLE}^{n+\nicefrac12}_{i+\nicefrac12}
  \cdot 2 D_h \tilde{x}^{n+1}_{i+\nicefrac12}
   =   2 f_0 (X_{i+\nicefrac12}) \Big( \xi^{(3)}_{i+\nicefrac12}
   ( D_h \tilde{x}^{n+1}_{i+\nicefrac12}  )^2
  + \xi^{(4)}_{i+\nicefrac12} D_h \tilde{x}^{n+1}_{i+\nicefrac12}
  \cdot D_h \tilde{x}^n_{i+\nicefrac12}  \Big)  \nonumber
\\
  &\ge&
  2 f_0 (X_{i+\nicefrac12}) \Big( \frac{1}{2 \tilde{C}^*} h
     ( D_h \tilde{x}^{n+1}_{i+\nicefrac12}  )^2
  -  2 ( \epsilon_0^*)^{-2} D_h \tilde{x}^{n+1}_{i+\nicefrac12}
  \cdot D_h \tilde{x}^n_{i+\nicefrac12}  \Big)  \nonumber
\\
  &\ge&
  2 f_0 (X_{i+\nicefrac12}) \cdot \Big(
  -  \frac{2 \tilde{C}^* ( \epsilon_0^*)^{-4} }{h}
  \cdot ( D_h \tilde{x}^n_{i+\nicefrac12} )^2  \Big)  ,
   \label{convergence-rough-7-1}
\end{eqnarray}
in which the Cauchy inequality has been applied at the last step. A summation of this inequality yields
\begin{eqnarray}
  \Big\langle f_0 (X) {\cal NLE}^{n+\nicefrac12} ,  2 D_h \tilde{x}^{n+1}  \Big\rangle
  \ge - \tilde{C}_7 h^{-1} \| D_h \tilde{x}^n \|_2^2 ,   \quad \mbox{with} \, \, \,
  \tilde{C}_7 = 4 \tilde{C}^* ( \epsilon_0^*)^{-4} \| f_0 (X) \|_\infty .
   \label{convergence-rough-7-2}
\end{eqnarray}

Finally, a substitution of~\eqref{convergence-rough-2-2}, \eqref{convergence-rough-2-3}, \eqref{convergence-rough-3}, \eqref{convergence-rough-4}, \eqref{convergence-rough-5-1}, \eqref{convergence-rough-5-2}, \eqref{convergence-rough-6-1} and \eqref{convergence-rough-7-2} into~\eqref{convergence-rough-1} leads to
\begin{eqnarray}
   &&
   \Bigl\langle \frac{ (S_h (x^n , x^{n-1} )  )^{m-1} }{m ( f_0(X) )^{m-2} }
   - \dt \Big( \frac{\tilde{C}_6}{2}  + 1 + \tilde{C}_5 \dt  \Big) ,
    ( \tilde{x}^{n+1} )^2 \Bigr\rangle  \nonumber
\\
  &\le&
     \Bigl\langle \frac{ (S_h (x^{n-1} , x^{n-2} )  )^{m-1} }{m ( f_0(X) )^{m-2} }  ,
    ( \tilde{x}^n )^2 \Bigr\rangle  \Bigr)
    + \dt \tilde{C}_4 \| \tilde{x}^n \|_2^2
    + ( A_0 + 4 (\epsilon_0^*)^{-4} ) \dt^2 \| D_x \tilde{x}^n \|_2^2  \nonumber
\\
  &&
    + \dt \| \tau^n \|_2^2
    + \frac{\tilde{C}_6}{2} \dt \| D_h \breve{\tilde{x}}^{n+\nicefrac12} \|_2^2
    + \tilde{C}_7 \cdot \frac{\dt}{h} \| D_h \tilde{x}^n \|_2^2 .
    \label{convergence-rough-8-1}
\end{eqnarray}
On the left hand side, we observe the following point-wise lower bound, which comes from the preliminary estimate~\eqref{prelim bound-5}:
\begin{eqnarray}
  &&
  \frac{ (S_h (x^n , x^{n-1} )  )^{m-1} }{m ( f_0(X) )^{m-2} }
  \ge \tilde{C}_8 := \frac1m (\frac{\epsilon_0^*}{4})^{m-1}
   \min \Big( \frac{1}{\max_\Omega f_0 (X)} , \min_\Omega f_0 (X) \Big)^{|m-2|} ,
   \label{convergence-rough-8-2}
\\
  &&
  \dt \Big( \frac{\tilde{C}_6}{2}  + 1 + \tilde{C}_5 \dt  \Big) \le \frac{\tilde{C}_8}{2} ,
  \quad \mbox{provided that $\dt$ is sufficiently small} ,
    \label{convergence-rough-8-3}
\end{eqnarray}
which in turn indicates that
\begin{eqnarray}
  \Bigl\langle \frac{ (S_h (x^n , x^{n-1} )  )^{m-1} }{m ( f_0(X) )^{m-2} }
   - \dt \Big( \frac{\tilde{C}_6}{2}  + 1 + \tilde{C}_5 \dt  \Big) ,
    ( \tilde{x}^{n+1} )^2 \Bigr\rangle
   \ge \frac{\tilde{C}_8}{2} \| \tilde{x}^{n+1} \|_2^2 .
    \label{convergence-rough-8-4}
\end{eqnarray}
On the right hand side, the following estimates are available, based on the a-priori assumption~\eqref{a priori-1} and the preliminary estimate~\eqref{prelim bound-1}:
\begin{eqnarray}
  &&
  \Bigl\langle \frac{ (S_h (x^{n-1} , x^{n-2} )  )^{m-1} }{m ( f_0(X) )^{m-2} }  ,
    ( \tilde{x}^n )^2 \Bigr\rangle  \Bigr)
    = O ( (\dt^4 + h^4)^2) ,   \label{convergence-rough-8-5}
\\
  &&
  \dt \tilde{C}_4 \| \tilde{x}^n \|_2^2 + \dt \| \tau^n \|_2^2
  = O ( \dt (\dt^4 + h^4)^2) ,   \label{convergence-rough-8-6}
\\
  &&
  ( A_0 + 4 (\epsilon_0^*)^{-4} ) \dt^2 \| D_x \tilde{x}^n \|_2^2
  = O ( \dt^2 (\dt^3 + h^3)^2) ,   \label{convergence-rough-8-7}
\\
  &&
  \frac{\tilde{C}_6}{2} \dt \| D_h \breve{\tilde{x}}^{n+\nicefrac12} \|_2^2
  = O ( \dt (\dt^3 + h^3)^2) ,   \label{convergence-rough-8-8}
\\
  &&
  \tilde{C}_7 \cdot \frac{\dt}{h} \| D_h \tilde{x}^n \|_2^2
  \le   C \tilde{C}_7 C_2 {\cal C}^2 ( \dt^3 + h^3 )^2 .
       \label{convergence-rough-8-9}
\end{eqnarray}
Then we arrive at a rough estimate for the numerical error function at time step $t^{n+1}$:
\begin{eqnarray}
   &&
   \frac{\tilde{C}_8}{2} \| \tilde{x}^{n+1} \|_2^2
   \le \Big( C \tilde{C}_7 C_2 {\cal C}^2 + 1 \Big) ( \dt^3 + h^3 )^2 ,  \quad
   \mbox{provided that $\dt$ and $h$ are sufficiently small} ,  \nonumber
 \\
   &&  \mbox{i.e.} \, \, \,
   \| \tilde{x}^{n+1} \|_2
   \le \hat{C} ( \dt^3 + h^3 ) ,  \quad \mbox{with} \, \, \,
   \hat{C} :=
   \Big( \frac{2 ( C \tilde{C}_7 C_2 {\cal C}^2 + 1)}{\tilde{C}_8} \Big)^\frac12 , 
   \label{convergence-rough-8-10}
\end{eqnarray}
under the linear refinement requirement $C_1 h \le \dt \le C_2 h$. Subsequently, an application of 1-D inverse inequality implies that
\begin{eqnarray}
  &&
   \| D_h \tilde{x}^{n+1} \|_\infty
  \le \frac{C \| \tilde{x}^{n+1} \|_2 }{h^\frac32}
  \le \hat {C}_1 ( \dt^\frac32 + h^\frac32 )  ,   \quad
  \mbox{with $\hat{C}_1 = C \hat{C}$} , \label{convergence-rough-9-1}
\end{eqnarray}
under the same linear refinement requirement. Because of the accuracy order, we could take $\dt$ and $h$ sufficient small so that $\hat {C}_1 ( \dt^\frac32 + h^\frac32 ) \le \frac{\epsilon_0^*}{2}$, which in turn gives
\begin{eqnarray}
  \frac{\epsilon_0^*}{2}  \le D_h x^{n+1}  =  D_h W^{n+1} -  D_h \tilde{x}^{n+1}
  \le C^* + \frac{\epsilon_0^*}{2} = \tilde{C}^* ,  \label{convergence-rough-10}
\end{eqnarray}
in which the lower and upper bounds~\eqref{assumption:separation-2}, \eqref{assumption:constructed bound} (for the constructed profile $W$) have been used again. Such a uniform $\| \cdot \|_{W_h^{1,\infty}}$ bound will play a very important role in the refined error estimate.

\begin{rem}
In the rough error estimate~\eqref{convergence-rough-8-10}, we see that the accuracy order is lower than the one given by the a-priori-assumption~\eqref{a priori-1}. Therefore, such a rough estimate could not be used for a global induction analysis. Instead, the purpose of such an estimate is to establish a uniform $\| \cdot \|_{W_h^{1,\infty}}$ bound for the numerical solution at time step $t^{n+1}$, as well as its temporal derivative, via the technique of inverse inequality. With these bounds established for the numerical solution, the refined error analysis will yield much sharper estimates.
\end{rem}

\subsection{A further rough error estimate}

Meanwhile, we have to derive a discrete $W_h^{1,\infty}$ bound for the second order temporal derivative of the numerical solution at time step $t^{n+1}$, which will be needed in the refined error estimate. In fact, such a bound could not be obtained by~\eqref{convergence-rough-9-1}. To obtain such a bound, we have to perform a further rough error estimate.

  We revisit the proof of Lemma~\ref{lem:nonlinear-1} and discover that, $\xi^{(3)}_{i+\nicefrac12}$ has to be between $D_h W^{n+1}$ and $D_h x^{n+1}$, based on the representation~\eqref{nonlinear est-1-1} and the property~\eqref{q convexity-2}. In more details, the regularity assumption~\eqref{assumption:constructed bound} (for $W$) and the $W_h^{1\infty}$ bound~\eqref{convergence-rough-10} imply a similar bound for $\xi^{(3)}_{i+\nicefrac12}$:
\begin{eqnarray}
  \frac{1}{2 (\tilde{C}^*)^2} \le \xi^{(4)}_{i+\nicefrac12} \le 2 ( \epsilon_0^*)^{-2}  .  \label{convergence-rough-11}
\end{eqnarray}
With such a bound at hand, we are able to rewrite the inner product in a more precise way:
\begin{eqnarray}
  &&
  f_0 (X_{i+\nicefrac12}) {\cal NLE}^{n+\nicefrac12}_{i+\nicefrac12}
  \cdot 2 D_h \tilde{x}^{n+1}_{i+\nicefrac12}
   =   2 f_0 (X_{i+\nicefrac12}) \Big( \xi^{(3)}_{i+\nicefrac12}
   ( D_h \tilde{x}^{n+1}_{i+\nicefrac12}  )^2
  + \xi^{(4)}_{i+\nicefrac12} D_h \tilde{x}^{n+1}_{i+\nicefrac12}
  \cdot D_h \tilde{x}^n_{i+\nicefrac12}  \Big)  \nonumber
\\
  &\ge&
  2 f_0 (X_{i+\nicefrac12}) \Big( \frac{1}{2 (\tilde{C}^*)^2}
     ( D_h \tilde{x}^{n+1}_{i+\nicefrac12}  )^2
  -  2 ( \epsilon_0^*)^{-2} D_h \tilde{x}^{n+1}_{i+\nicefrac12}
  \cdot D_h \tilde{x}^n_{i+\nicefrac12}  \Big)  \nonumber
\\
  &\ge&
  2 f_0 (X_{i+\nicefrac12}) \cdot \Big(
  \frac{1}{4 (\tilde{C}^*)^2}   ( D_h \tilde{x}^{n+1}_{i+\nicefrac12}  )^2
  -  4 (\tilde{C}^*)^2 ( \epsilon_0^*)^{-4}
  \cdot ( D_h \tilde{x}^n_{i+\nicefrac12} )^2  \Big)  ,   \quad
  \mbox{which in turn gives}   \nonumber 
\\
  &&
  \Big\langle f_0 (X) {\cal NLE}^{n+\nicefrac12} ,  2 D_h \tilde{x}^{n+1}  \Big\rangle
  \ge \tilde{C}_9 \| \tilde{x}^{n+1} \|_2^2 - \tilde{C}_{10} \| D_h \tilde{x}^n \|_2^2 ,      \label{convergence-rough-12-1}
\\
  &&
  \mbox{with} \, \, \,
  \tilde{C}_9 = \frac{1}{2 (\tilde{C}^*)^2}  \min_\Omega ( f_0 (X)) ,  \, \, \,
  \tilde{C}_{10} = 8 (\tilde{C}^*)^2 ( \epsilon_0^*)^{-4} \| f_0 (X) \|_\infty .
   \nonumber
\end{eqnarray}
A substitution of this updated estimate yields a rewritten inequality for~\eqref{convergence-rough-8-1}:
\begin{eqnarray}
   &&
   \Bigl\langle \frac{ (S_h (x^n , x^{n-1} )  )^{m-1} }{m ( f_0(X) )^{m-2} }
   - \dt \Big( \frac{\tilde{C}_6}{2}  + 1 + \tilde{C}_5 \dt  \Big) ,
    ( \tilde{x}^{n+1} )^2 \Bigr\rangle
    + \tilde{C}_9 \dt \| D_h  \tilde{x}^{n+1} \|_2^2  \nonumber
\\
  &\le&
     \Bigl\langle \frac{ (S_h (x^{n-1} , x^{n-2} )  )^{m-1} }{m ( f_0(X) )^{m-2} }  ,
    ( \tilde{x}^n )^2 \Bigr\rangle  \Bigr)
    + \dt \tilde{C}_4 \| \tilde{x}^n \|_2^2
    + ( A_0 + 4 (\epsilon_0^*)^{-4} ) \dt^2 \| D_x \tilde{x}^n \|_2^2  \nonumber
\\
  &&
    + \dt \| \tau^n \|_2^2
    + \frac{\tilde{C}_6}{2} \dt \| D_h \breve{\tilde{x}}^{n+\nicefrac12} \|_2^2
    + \tilde{C}_{10} \dt \| D_h \tilde{x}^n \|_2^2 .
    \label{convergence-rough-13-1}
\end{eqnarray}
Meanwhile, all other estimates~\eqref{convergence-rough-8-2}-\eqref{convergence-rough-8-8} are still valid, then we arrive at
\begin{eqnarray}
  \frac{\tilde{C}_8}{2} \| \tilde{x}^{n+1} \|_2^2
  + \tilde{C}_9 \dt \| D_h  \tilde{x}^{n+1} \|_2^2
  &\le& \frac{\tilde{C}_6}{2} \dt \| D_h \breve{\tilde{x}}^{n+\nicefrac12} \|_2^2
  + \tilde{C}_{10} \dt \| D_h \tilde{x}^n \|_2^2  + O (\dt^2 (\dt^3 + h^3)^2)
  \nonumber
\\
  &\le&
  \tilde{C}_{11} \dt ( \| D_h \tilde{x}^n \|_2^2  + \| D_h \tilde{x}^{n-1} \|_2^2 )
  + O (\dt^2 (\dt^3 + h^3)^2)   \nonumber
\\
  &\le&
  \tilde{C}_{12} \dt (\dt^3 + h^3)^2 , \label{convergence-rough-13-1}
\end{eqnarray}
with $\tilde{C}_{11} = \tilde{C}_9 + \frac{9 \tilde{C}_6}{4}$, $\tilde{C}_{12} = C \tilde{C}_{11} {\cal C}^2 + 1$, provided that $\dt$ and $h$ are sufficiently small. This in turn results in a further rough estimate for $\tilde{x}^{n+1}$:
\begin{eqnarray}
    \| D_h  \tilde{x}^{n+1} \|_2   \le \hat{C}_2 ( \dt^3 + h^3) ,  \quad \mbox{with} \, \, \,
   \hat{C}_2 :=  \Big( \frac{\tilde{C}_{12} }{\tilde{C}_9} \Big)^\frac12 .
   \label{convergence-rough-13-2}
\end{eqnarray}
As a consequence, an application of 1-D inverse inequality gives a sharper estimate for $\| D_h \tilde{x}^{n+1} \|_\infty$: 
\begin{eqnarray}
  &&
   \| D_h \tilde{x}^{n+1} \|_\infty
  \le \frac{C \| D_h \tilde{x}^{n+1} \|_2 }{h^\frac12}
  \le \hat {C}_3 ( \dt^\frac52 + h^\frac52 )  ,   \quad
  \mbox{with $\hat{C}_3 = C \hat{C}_2$} , \label{convergence-rough-13-3}
\\
  &&
   \| D_h D_t^2 \tilde{x}^n  \|_\infty
  \le ( \hat {C}_3 + 1) ( \dt^\frac32 + h^\frac32 )
  \le \dt ,   \, \, \, D_t^2 \tilde{x}^n := \frac{\tilde{x}^{n+1}
  - 2 \tilde{x}^n + \tilde{x}^{n-1}}{\dt^2}   , \label{convergence-rough-13-4}
\end{eqnarray}
in combination with~\eqref{prelim bound-2} , under the same linear refinement requirement. This $\| \cdot \|_{W_h^{1,\infty}}$ bound for the second order temporal derivative will play a very important role in the refined error estimate.

\subsection{The refined error estimate}

Now we proceed with the refined error estimate. Taking a discrete inner product with~\eqref{error equation-1} by $2 \tilde{x}^{n+\nicefrac12} = \tilde{x}^{n+1} + \tilde{x}^n$ leads to
\begin{eqnarray}
  &&
   \Bigl\langle \frac{ (S_h (x^n , x^{n-1} )  )^{m-1} }{m ( f_0(X) )^{m-2} }
   \cdot \frac{\tilde{x}^{n+1} - \tilde{x}^n}{\dt} , \tilde{x}^{n+1} + \tilde{x}^n \Bigr\rangle
   +  A_0 \dt \langle D_h ( \tilde{x}^{n+1} - \tilde{x}^n ) ,
   D_h ( \tilde{x}^{n+1}  + \tilde{x}^n ) \rangle
   \nonumber
\\
  &&
  - \Bigl\langle  f_{0}(X) \Big(
    \frac{\ln (D_h W^{n+1}) - \ln ( D_h W^n) }{D_h ( W^{n+1} - W^n ) }
  - \frac{\ln (D_h x^{n+1}) - \ln ( D_h x^n) }{D_h ( x^{n+1} - x^n ) }  \Big) ,
  D_h ( \tilde{x}^{n+1} + \tilde{x}^n ) \Bigr\rangle \nonumber
\\
  &=&
   - \Bigl\langle \frac{ (S_h (W^n , W^{n-1} )  )^{m-1} - (S_h (x^n , x^{n-1} )  )^{m-1} }{m ( f_0(X) )^{m-2} }  \cdot \frac{W^{n+1} - W^{n}}{\dt}  ,
     2 \tilde{x}^{n+\nicefrac12} \Bigr\rangle  \nonumber
\\
  &&
  - \dt^2 \Bigl\langle \frac{D_h \tilde{x}^{n+1} }{ D_h W^{n+1} D_h x^{n+1}}
  - \frac{D_h \tilde{x}^n}{D_h W^n D_h x^n}  ,
   D_h ( \tilde{x}^{n+1} + \tilde{x}^n ) \Bigr\rangle
  + \langle \tau^n ,  \tilde{x}^{n+1} + \tilde{x}^n  \rangle .
  \label{convergence-1}
\end{eqnarray}

For the temporal derivative term, the equality $(\tilde{x}^{n+1} + \tilde{x}^n ) (\tilde{x}^{n+1} - \tilde{x}^n) = (\tilde{x}^{n+1})^2 - (\tilde{x}^n)^2$ implies a similar estimate as in~\eqref{convergence-rough-2-2}-\eqref{convergence-rough-2-3}:
\begin{eqnarray}
  &&
  \Bigl\langle \frac{ (S_h (x^n , x^{n-1} )  )^{m-1} }{m ( f_0(X) )^{m-2} }
   \cdot \frac{\tilde{x}^{n+1} - \tilde{x}^n}{\dt} ,
   \tilde{x}^{n+1} + \tilde{x}^n \Bigr\rangle   \nonumber
 \\
   &=&
   \frac{1}{\dt}  \Bigl(
   \Bigl\langle \frac{ (S_h (x^n , x^{n-1} )  )^{m-1} }{m ( f_0(X) )^{m-2} }  ,
    ( \tilde{x}^{n+1} )^2 \Bigr\rangle
   - \Bigl\langle \frac{ (S_h (x^n , x^{n-1} )  )^{m-1} }{m ( f_0(X) )^{m-2} }  ,
    ( \tilde{x}^n )^2 \Bigr\rangle  \Bigr)   \nonumber
\\
  &\ge&
  \frac{1}{\dt}  \Bigl(
   \Bigl\langle \frac{ (S_h (x^n , x^{n-1} )  )^{m-1} }{m ( f_0(X) )^{m-2} }  ,
    ( \tilde{x}^{n+1} )^2 \Bigr\rangle
   - \Bigl\langle \frac{ (S_h (x^{n-1} , x^{n-2} )  )^{m-1} }{m ( f_0(X) )^{m-2} }  ,
    ( \tilde{x}^n )^2 \Bigr\rangle  \Bigr)
    - \tilde{C}_4 \| \tilde{x}^n \|_2^2  .    \label{convergence-2}
\end{eqnarray}

For the second term on the left hand side, the second term on the right hand side of~\eqref{convergence-1}, and the local truncation error terms, the following bounds could be similarly derived:
\begin{eqnarray}
  &&
   A_0 \dt \langle D_h ( \tilde{x}^{n+1} - \tilde{x}^n ) ,
   D_h ( \tilde{x}^{n+1} + \tilde{x}_n )  \rangle
   =  A_0 \dt ( \| D_h  \tilde{x}^{n+1} \|_2^2  - \| D_h \tilde{x}^n \|_2 ) ,
     \label{convergence-3}
\\
  &&
  2 \langle \tau^n ,  \tilde{x}^{n+1}  + \tilde{x}^n  \rangle
   \le \| \tau^n \|_2^2 +  \| \tilde{x}^{n+\nicefrac12} \|_2^2
   \le \| \tau^n \|_2^2 +  \frac12 ( \| \tilde{x}^{n+1} \|_2^2
   + \| \tilde{x}^n \|_2^2 ) ,   \label{convergence-4}
\\
  &&
  - \dt^2 \Bigl\langle \frac{D_h \tilde{x}^{n+1} }{ D_h W^{n+1} D_h x^{n+1}}
  - \frac{D_h \tilde{x}^n}{D_h W^n D_h x^n}  ,
   D_h ( \tilde{x}^{n+1} + \tilde{x}^n ) \Bigr\rangle  \nonumber
\\
        &\le& 2 \dt^2 \cdot \frac{1}{\frac12 (\epsilon_0^*)^2}
     \| D_h \tilde{x}^n \|_2 \cdot \| D_h \tilde{x}^{n+1} \|_2
     - \dt^2 \cdot \frac{1}{\frac12 (\epsilon_0^*)^2}
     \| D_h \tilde{x}^n \|_2^2   \nonumber
\\
  &\le&
  2 \dt^2 (\epsilon_0^*)^{-2} ( \| D_h \tilde{x}^{n+1} \|_2^2
  + \| D_h \tilde{x}^n \|_2^2 )
  + 2 \dt^2 (\epsilon_0^*)^{-2}  \| D_h \tilde{x}^n \|_2^2   \nonumber
\\
  &\le&
    2 \dt^2 (\epsilon_0^*)^{-2} ( \| D_h \tilde{x}^{n+1} \|_2^2
  + 2 \| D_h \tilde{x}^n \|_2^2 ) ,
   \label{convergence-5}
\end{eqnarray}
in which the separation property~\eqref{assumption:separation-2} (for the constructed profile $W$), the preliminary estimates~\eqref{prelim bound-3}, \eqref{convergence-rough-10} (for $x^n$ and $x^{n+1}$, respectively), have been used in~\eqref{convergence-5}.

For the first term on the right hand side of~\eqref{convergence-1}, we cannot count on the estimate~\eqref{convergence-rough-6-1}, since there is no stability control for  $\| D_h \breve{\tilde{x}}^{n+\nicefrac12}  \|_2^2 = \| D_h ( \frac32 \tilde{x}^n - \frac12 \tilde{x}^{n-1} ) \|_2^2$.   
To overcome this difficulty, a summation by parts formula is applied, due to the fact that $\tilde{x}^{n+\nicefrac12}_0 = \tilde{x}^{n+\nicefrac12}_M =0$:
\begin{eqnarray}
  &&
   - 2 \Bigl\langle \frac{ (S_h (W^n , W^{n-1} )  )^{m-1} - (S_h (x^n , x^{n-1} )  )^{m-1} }{m ( f_0(X) )^{m-2} }  \cdot \frac{W^{n+1} - W^{n}}{\dt}  ,
     \tilde{x}^{n+\nicefrac12} \Bigr\rangle   \nonumber
\\
  &=&
   - 2 \Bigl\langle \frac{ {\cal N}^{n+\nicefrac12}  \widetilde{D}_h \breve{\tilde{x}}^{n+\nicefrac12} }{m ( f_0(X) )^{m-2} }  \cdot \frac{W^{n+1} - W^{n}}{\dt}  ,
     \tilde{x}^{n+\nicefrac12} \Bigr\rangle    \nonumber
\\
  &=&
    2 \Bigl\langle \widetilde{D}_h \Bigl( \frac{ {\cal N}^{n+\nicefrac12}  \frac{W^{n+1} - W^{n}}{\dt} }{m ( f_0(X) )^{m-2} }  \cdot \tilde{x}^{n+\nicefrac12}  \Bigr)  ,
       \breve{\tilde{x}}^{n+\nicefrac12}   \Bigr\rangle .   \label{convergence-6-1}
\end{eqnarray}
Meanwhile, the following observation is made in the finite difference space:
\begin{eqnarray}
  \hspace{-0.4in} &&
  \nrm{ \widetilde{D}_h \Bigl( \frac{ {\cal N}^{n+\nicefrac12}  \frac{W^{n+1} - W^{n}}{\dt} }{m ( f_0(X) )^{m-2} }  \cdot \tilde{x}^{n+\nicefrac12}  \Bigr)  }_2
  \le \nrm{ D_h \Bigl( \frac{ {\cal N}^{n+\nicefrac12}  \frac{W^{n+1} - W^{n}}{\dt} }{m ( f_0(X) )^{m-2} }  \cdot \tilde{x}^{n+\nicefrac12} \Bigr) }_2   \nonumber
\\
  \hspace{-0.4in} &\le&
   \frac1m  ( \|  \tilde{x}^{n+\nicefrac12}  \|_2 +  \| D_h \tilde{x}^{n+\nicefrac12}  \|_2 )
  \cdot (  \|  {\cal N}^{n+\nicefrac12} \|_\infty
  +  \| D_h {\cal N}^{n+\nicefrac12}  \|_\infty ) \nonumber
\\
  \hspace{-0.4in} &&
  \cdot \Big(  \| \frac{W^{n+1} - W^{n}}{\dt}  \|_\infty
  + \| \frac{D_h ( W^{n+1} - W^{n}) }{\dt}  \|_\infty \Big)   \nonumber
\\
   \hspace{-0.4in}  &&
  \cdot \Bigl(  \| \frac{1}{( f_0 (X) )^{m-2} }  \|_\infty
  + \|  D_h \Big( \frac{1}{( f_0 (X) )^{m-2} }  \Big) \|_\infty \Big)   \nonumber
\\
  \hspace{-0.4in} &\le&
   \frac1m  (\tilde{C}_2 + \tilde{C}_3 ) \tilde{C}_{13} C^*
   ( \|  \tilde{x}^{n+\nicefrac12}  \|_2 +  \| D_h \tilde{x}^{n+\nicefrac12}  \|_2 ) ,
   \label{convergence-6-2}
\\
  &&
   \mbox{with} \, \, \, \tilde{C}_{13} := \| \frac{1}{( f_0 (X) )^{m-2} }  \|_\infty
  + \| D_h \Big( \frac{1}{( f_0 (X) )^{m-2} }  \Big) \|_\infty ,  \nonumber
\end{eqnarray}
in which the preliminary $W_h^{1,\infty}$ estimates~\eqref{prelim est-2}, \eqref{prelim est-3}, and the regularity assumption~\eqref{assumption:constructed bound} have been repeatedly used in the derivation. Then we arrive at
\begin{eqnarray}
  &&
   - 2 \Bigl\langle \frac{ (S_h (W^n , W^{n-1} )  )^{m-1} - (S_h (x^n , x^{n-1} )  )^{m-1} }{m ( f_0(X) )^{m-2} }  \cdot \frac{W^{n+1} - W^{n}}{\dt}  ,
     \tilde{x}^{n+\nicefrac12} \Bigr\rangle   \nonumber
\\
  &\le&
   2 \tilde{C}_{14}
   ( \|  \tilde{x}^{n+\nicefrac12}  \|_2 +  \| D_h \tilde{x}^{n+\nicefrac12}  \|_2 )
   \| \breve{\tilde{x}}^{n+\nicefrac12} \|_2 ,   \quad
   (\mbox{with} \, \, \,
   \tilde{C}_{14} =  \frac1m  (\tilde{C}_2 + \tilde{C}_3 ) \tilde{C}_9 C^*) \nonumber
\\
  &\le&
  \Bigl( \tilde{C}_{14} + \tilde{C}_{14}^2 ( \tilde{C}^* )^2  \Big)
  \| \breve{\tilde{x}}^{n+\nicefrac12} \|_2^2
  + \tilde{C}_{14} \|  \tilde{x}^{n+\nicefrac12}  \|_2^2
  +  (\tilde{C}^*)^{-2} \| D_h \tilde{x}^{n+\nicefrac12}  \|_2^2 .
     \label{convergence-6-3}
\end{eqnarray}

Again, the rest work is focused on the error analysis associated with the nonlinear diffusion part, as given by~\eqref{notation-3}. However, the point-wise estimate~\eqref{nonlinear est-0-1}, \eqref{nonlinear est-0-2} is not useful in the refined analysis any more. Instead, we begin with the following application of higher order Taylor expansion for $\ln x$, around $\frac{x + x_0}{2}$:
\begin{eqnarray}
  &&
   \frac{\ln x - \ln x_0 }{x - x_0 }
   = \frac{1}{\frac{x + x_0}{2} }
   + \frac{2}{3  ( \frac{x + x_0}{2} )^3 } \cdot \frac{(x - x_0)^2}{8}
   +  \Big( \frac{1}{5  ( \xi^{(5)} )^5 } + \frac{1}{5  ( \xi^{(6)} )^5 }  \Big)
   \cdot \frac{(x - x_0)^4}{32}  ,  \label{Taylor-1}
\end{eqnarray}
with $\xi^{(5)}$ between $\frac{x+x_0}{2}$ and $x$, $\xi^{(6)}$ between $\frac{x+x_0}{2}$ and $x_0$. This in turn gives
\begin{eqnarray}
   \frac{\ln (D_h W^{n+1}) - \ln ( D_h W^n) }{D_h ( W^{n+1} - W^n ) }
   &=&  \frac{1}{D_h ( \frac{W^{n+1} + W^n}{2} ) }
   + \frac{2}{3  ( D_h \frac{ W^{n+1} + W^n}{2} )^3 }
   \cdot \frac{(D_h ( W^{n+1} - W^n) )^2}{8}   \nonumber
\\
  &&
   +  \Big( \frac{1}{5  ( \eta^{(1)} )^5 } + \frac{1}{5  ( \eta^{(2)} )^5 }  \Big)
   \cdot \frac{(D_h ( W^{n+1} - W^n) )^4}{32} ,
    \label{convergence-7-1}
\\
  \frac{\ln (D_h x^{n+1}) - \ln ( D_h x^n) }{D_h ( x^{n+1} - x^n ) }
   &=&  \frac{1}{D_h ( \frac{x^{n+1} + x^n}{2} ) }
   + \frac{2}{3  ( D_h \frac{ x^{n+1} + x^n}{2} )^3 }
   \cdot \frac{(D_h ( x^{n+1} - x^n) )^2}{8}   \nonumber
\\
  &&
   +  \Big( \frac{1}{5  ( \eta^{(3)} )^5 } + \frac{1}{5  ( \eta^{(4)} )^5 }  \Big)
   \cdot \frac{(D_h ( x^{n+1} - x^n) )^4}{32} ,
    \label{convergence-7-2}
\end{eqnarray}
with
\begin{eqnarray}
  &&
  \eta^{(1)} \, \, \mbox{between} \, \, D_h ( \frac{W^{n+1} + W^n}{2} ) \, \,
  \mbox{and} \, \, D_h W^{n+1} ,  \quad
  \eta^{(2)} \, \, \mbox{between} \, \, D_h ( \frac{W^{n+1} + W^n}{2} ) \, \,
  \mbox{and} \, \, D_h W^n ,   \nonumber
\\
  &&
  \eta^{(3)} \, \, \mbox{between} \, \, D_h ( \frac{x^{n+1} + x^n}{2} ) \, \,
  \mbox{and} \, \, D_h x^{n+1} ,  \quad
  \eta^{(4)} \, \, \mbox{between} \, \, D_h ( \frac{x^{n+1} + x^n}{2} ) \, \,
  \mbox{and} \, \, D_h x^n .   \nonumber
\end{eqnarray}
Then we arrive at a decomposition for ${\cal NLE}^{n+\nicefrac12}$:
\begin{eqnarray}
  {\cal NLE}^{n+\nicefrac12} &=&
  {\cal NL}^{(1)} + {\cal NL}^{(2)} + {\cal NL}^{(3)} ,  \quad \mbox{with} \nonumber
\\
  {\cal NL}^{(1)} &=&  - \frac{1}{D_h ( \frac{W^{n+1} + W^n}{2} ) }
  +  \frac{1}{D_h ( \frac{x^{n+1} + x^n}{2} ) }
  =  \frac{D_h \tilde{x}^{n+\nicefrac12} }{D_h ( \frac{W^{n+1} + W^n}{2} ) \cdot D_h ( \frac{x^{n+1} + x^n}{2} ) }  ,  \label{convergence-7-3}
\\
  {\cal NL}^{(2)} &=&  \frac{1}{12} \Big(
  - \frac{(D_h ( W^{n+1} - W^n) )^2}{( D_h \frac{ W^{n+1} + W^n}{2} )^3 }
  + \frac{(D_h ( x^{n+1} - x^n) )^2}{( D_h \frac{ x^{n+1} + x^n}{2} )^3 }  \Big) ,
  \label{convergence-7-4}
\\
  {\cal NL}^{(3)} &=&  \frac{1}{160} \Big(
  - \Big( \frac{1}{( \eta^{(1)} )^5 } + \frac{1}{( \eta^{(2)} )^5 }  \Big)
  (D_h ( W^{n+1} - W^n) )^4    \nonumber
\\
  &&
  + \Big( \frac{1}{( \eta^{(3)} )^5 } + \frac{1}{( \eta^{(4)} )^5 }  \Big)
  (D_h ( x^{n+1} - x^n) )^4 \Big) .
  \label{convergence-7-5}
\end{eqnarray}
For the leading expansion ${\cal NL}^{(1)}$, the following nonlinear estimate is available:
\begin{eqnarray}
  \langle  {\cal NL}^{(1)} , 2 D_h \tilde{x}^{n+\nicefrac12} \rangle
  &=& \Big\langle \frac{ 2 }{D_h ( \frac{W^{n+1} + W^n}{2} ) \cdot D_h ( \frac{x^{n+1} + x^n}{2} ) }  ,  ( D_h \tilde{x}^{n+\nicefrac12} )^2 \Big\rangle   \nonumber
\\
  &\ge&
  \frac{2}{( \tilde{C}^*)^2 }  \|  D_h \tilde{x}^{n+\nicefrac12}  \|_2^2 ,
  \label{convergence-7-6}
\end{eqnarray}
with repeated applications of~\eqref{prelim est-2}, \eqref{prelim est-3} and~\eqref{assumption:constructed bound}. The second expansion ${\cal NL}^{(2)}$ could have a further decomposition:  ${\cal NL}^{(2)} = {\cal NL}^{(2),1}  + {\cal NL}^{(2),2}$, with
\begin{eqnarray}
  {\cal NL}^{(2),1} &=& - \frac{ D_h ( W^{n+1} - W^n + x^{n+1} - x^n)
  \cdot D_h (\tilde{x}^{n+1} - \tilde{x}^n ) }{( D_h  W^{n+\nicefrac12}  )^3 }  ,
  \label{convergence-7-7}
\\
  {\cal NL}^{(2),2} &=&  {\cal NLC}^{(2)}
  \cdot D_h \tilde{x}^{n+\nicefrac12} \cdot ( D_h (x^{n+1} - x^n ) )^2 ,
  \label{convergence-7-8}
\\
  {\cal NLC}^{(2)} &=& \frac{ ( D_h W^{n+\nicefrac12} )^2 + ( D_h W^{n+\nicefrac12} ) ( D_h x^{n+\nicefrac12} ) + ( D_h  x^{n+\nicefrac12} )^2 ) }{( D_h W^{n+\nicefrac12} )^3
  ( D_h x^{n+\nicefrac12} )^3} .   \label{convergence-7-9}
\end{eqnarray}
We notice a bound for the nonlinear coefficient ${\cal NLC}^{(2)}$:
\begin{eqnarray}
  |  {\cal NLC}^{(2)} |  \le  \frac{3 ( \tilde{C}^* )^2 }{\frac12 ( \epsilon_0^*)^2 }
  = 6 ( \tilde{C}^* )^2 ( \epsilon_0^*)^{-2} := \tilde{C}_{15} ,  \quad
  \mbox{by~\eqref{assumption:separation-2}, \eqref{assumption:constructed bound}, \eqref{prelim bound-3}, \eqref{convergence-rough-10} } .  \label{convergence-7-10}
\end{eqnarray}
This in turn yields an estimate for the term associated with ${\cal NL}^{(2),2}$:
\begin{eqnarray}
   \langle  {\cal NL}^{(2), 1} , 2 D_h \tilde{x}^{n+\nicefrac12}  \rangle
   &\ge&  - 2 \tilde{C}_{16} \dt \| D_h \tilde{x}^{n+\nicefrac12}  \|_2^2  ,
   \quad \mbox{with $\tilde{C}_{16} := \tilde{C}_{15} (C^* + 1)^2$} ,
   \label{convergence-7-11}
\end{eqnarray}
in which a point-wise bound $\| D_x (x^{n+1} - x^n) \|_\infty   \le (C^* +1) \dt$ comes from
\eqref{assumption:constructed bound}, \eqref{prelim bound-3} and \eqref{convergence-rough-9-1}. For ${\cal NL}^{(2),1}$, we introduce the following discrete function
\begin{eqnarray}
  \gamma^{n+\nicefrac12} &:=& - \frac{ \frac{ D_h ( W^{n+1} - W^n ) }{\dt}
  + \frac{ D_h ( x^{n+1} - x^n) }{\dt} }{( D_h  W^{n+\nicefrac12}  )^3 }  ,   \quad
  \mbox{so that}
  \label{convergence-7-12}
\\
   \langle  {\cal NL}^{(2), 1} , D_h ( \tilde{x}^{n+1} + \tilde{x}^n )  \rangle
   &=& \dt \langle  \gamma^{n+\nicefrac12} ,
   ( D_h \tilde{x}^{n+1} )^2 - ( D_h \tilde{x}^n )  )^2 \rangle \nonumber
\\
  &=&
  \dt ( \langle  \gamma^{n+\nicefrac12} , ( D_h \tilde{x}^{n+1} )^2  \rangle
   - \langle  \gamma^{n-\nicefrac12} , ( D_h \tilde{x}^n )  )^2 \rangle )   \nonumber
\\
  &&
  - \dt \langle  \gamma^{n+\nicefrac12} - \gamma^{n-\nicefrac12} ,
  ( D_h \tilde{x}^n )  )^2 \rangle .  \label{convergence-7-13}
\end{eqnarray}
For the last correction term, the following observation is made:
\begin{eqnarray}
  \gamma^{n+\nicefrac12} - \gamma^{n-\nicefrac12}
  &=&   - \frac{ \dt D_h ( D_t^2 W^n)
  + \dt D_h ( D_t^2 x^n) }{( D_h  W^{n+\nicefrac12}  )^3 }
  \nonumber
\\
  &&
  + {\cal NLC}^{(2)} \cdot \Big( \frac{D_h ( W^{n+1} - W^n ) }{\dt}
  + \frac{ D_h ( x^{n+1} - x^n)}{\dt}  \Big)
  \cdot D_h \tilde{x}^{n+\nicefrac12}   .  \label{convergence-7-14}
\end{eqnarray}
Meanwhile, a $\| \cdot \|_{W_h^{1,\infty}}$ bound for $D_t^2 x^n$ is available, as a result of the further rough estimate~\eqref{convergence-rough-13-4} and the regularity assumption~\eqref{assumption:constructed bound} for W:
\begin{eqnarray}
  \| D_h ( D_t^2 x^n ) \|_\infty
  \le  \| D_h ( D_t^2 W^n ) \|_\infty  +   \| D_h ( D_t^2 \tilde{x}^n ) \|_\infty
  \le C^* + \frac{\epsilon_0^*}{2} = \tilde{C}^* .
  \label{convergence-7-15}
\end{eqnarray}
This in turn implies an $O(\dt)$ estimate for the first part, in combination with~\eqref{assumption:separation-2}
\begin{eqnarray}
    \dt \Big| \frac{ D_h ( D_t^2 W^n)
  + D_h ( D_t^2 x^n) }{( D_h  W^{n+\nicefrac12}  )^3 }   \Big|
  \le  \frac{C^* + \tilde{C}^*}{(\epsilon_0^*)^3 } \dt .
  \label{convergence-7-16}
\end{eqnarray}
A similar bound for the second part is also available, which comes from~\eqref{assumption:constructed bound}, \eqref{prelim bound-2}, \eqref{convergence-rough-13-3} and \eqref{convergence-7-10}:
\begin{eqnarray}
  &&
    \Big| {\cal NLC}^{(2)} \cdot \Big( \frac{D_h ( W^{n+1} - W^n ) }{\dt}
  + \frac{ D_h ( x^{n+1} - x^n)}{\dt}  \Big)
  \cdot D_h \tilde{x}^{n+\nicefrac12}  \Big|   \nonumber
\\
  &\le&
   (2 C^* +1 ) \tilde{C}_{15} ( \hat{C} + \hat{C}_3 ) (\dt^\frac52 + h^\frac52)
  \le \dt .   \label{convergence-7-17}
\end{eqnarray}
Therefore, an $O (\dt)$ bound for $\gamma^{n+\nicefrac12} - \gamma^{n-\nicefrac12}$
is obtained:
\begin{eqnarray}
  \|  \gamma^{n+\nicefrac12} - \gamma^{n-\nicefrac12} \|_\infty
  \le \tilde{C}_{17} \dt ,   \quad \mbox{with} \, \, \,
  \tilde{C}_{17}:= ( C^* + \tilde{C}^* )  (\epsilon_0^*)^{-3} + 1 ,
   \label{convergence-7-18}
\end{eqnarray}
so that the nonlinear inner product associated with ${\cal NL}^{(2),1}$ could be analyzed as follows:
\begin{eqnarray}
    \langle  {\cal NL}^{(2), 1} , D_h ( \tilde{x}^{n+1} + \tilde{x}^n )  \rangle
  &\ge&
  \dt ( \langle  \gamma^{n+\nicefrac12} , ( D_h \tilde{x}^{n+1} )^2  \rangle
   - \langle  \gamma^{n-\nicefrac12} , ( D_h \tilde{x}^n )  )^2 \rangle )    \nonumber
\\
  &&
  - \tilde{C}_{17} \dt^2 \| D_h \tilde{x}^n \|_2^2 .    \label{convergence-7-18}
\end{eqnarray}
Its combination with~\eqref{convergence-7-11} yields the nonlinear estimate for ${\cal NL}^{(2)}$:
\begin{eqnarray}
  \hspace{-0.4in}
    \langle  {\cal NL}^{(2)} , D_h ( \tilde{x}^{n+1} + \tilde{x}^n )  \rangle
  &\ge&
  \dt ( \langle  \gamma^{n+\nicefrac12} , ( D_h \tilde{x}^{n+1} )^2  \rangle
   - \langle  \gamma^{n-\nicefrac12} , ( D_h \tilde{x}^n )  )^2 \rangle )    \nonumber
\\
  \hspace{-0.4 in} &&
  - \tilde{C}_{18} \dt^2 ( \| D_h \tilde{x}^{n+1} \|_2^2 +  \| D_h \tilde{x}^n \|_2^2 ) ,   \quad
  \mbox{with $\tilde{C}_{18} = \tilde{C}_{16} + \tilde{C}_{17}$ }.
     \label{convergence-7-19}
\end{eqnarray}
The analysis for ${\cal NL}^{(3)}$ is similar for that of ${\cal NL}^{(2),2}$. We are able to obtain the following estimate; the technical details are skipped for the sake of brevity.
\begin{eqnarray}
    \langle  {\cal NL}^{(3)} , D_h ( \tilde{x}^{n+1} + \tilde{x}^n )  \rangle
  \ge
  - \tilde{C}_{19} \dt^2 ( \| D_h \tilde{x}^{n+1} \|_2^2 +  \| D_h \tilde{x}^n \|_2^2 )
  - ( \dt^4 + h^4 )^2 .
     \label{convergence-7-20}
\end{eqnarray}
A combination of~\eqref{convergence-7-6}, \eqref{convergence-7-19} and \eqref{convergence-7-20} results in an estimate for ${\cal NLE}^{n+\nicefrac12}$:
\begin{eqnarray}
    \langle  {\cal NLE}^{n+\nicefrac12} , 2 D_h \tilde{x}^{n+\nicefrac12}   \rangle
  &\ge&
  \frac{2}{( \tilde{C}^*)^2 }  \|  D_h \tilde{x}^{n+\nicefrac12}  \|_2^2
  +  \dt ( \langle  \gamma^{n+\nicefrac12} , ( D_h \tilde{x}^{n+1} )^2  \rangle
   - \langle  \gamma^{n-\nicefrac12} , ( D_h \tilde{x}^n )  )^2 \rangle )  \nonumber
\\
  &&
  - \tilde{C}_{20} \dt^2 ( \| D_h \tilde{x}^{n+1} \|_2^2 +  \| D_h \tilde{x}^n \|_2^2 )
  - ( \dt^4 + h^4 )^2 ,
     \label{convergence-7-21}
\end{eqnarray}
with $\tilde{C}_{20} = \tilde{C}_{18} + \tilde{C}_{19}$.

Finally, a substitution of~\eqref{convergence-2}-\eqref{convergence-5}, \eqref{convergence-6-3} and \eqref{convergence-7-21} into~\eqref{convergence-1} results in
\begin{eqnarray}
   &&
   \Bigl\langle \frac{ (S_h (x^n , x^{n-1} )  )^{m-1} }{m ( f_0(X) )^{m-2} } ,
    ( \tilde{x}^{n+1} )^2 \Bigr\rangle
    +  \frac{1}{( \tilde{C}^*)^2 }  \dt \|  D_h \tilde{x}^{n+\nicefrac12}  \|_2^2  \nonumber
\\
  &&
    + A_0 \dt^2 ( \| D_h \tilde{x}^{n+1} \|_2^2 -  \| D_h \tilde{x}^n \|_2^2 )
    + \dt^2 ( \langle  \gamma^{n+\nicefrac12} , ( D_h \tilde{x}^{n+1} )^2  \rangle
   - \langle  \gamma^{n-\nicefrac12} , ( D_h \tilde{x}^n )  )^2 \rangle )   \nonumber
\\
  &\le&
     \Bigl\langle \frac{ (S_h (x^{n-1} , x^{n-2} )  )^{m-1} }{m ( f_0(X) )^{m-2} }  ,
    ( \tilde{x}^n )^2 \Bigr\rangle  \Bigr)
    + \tilde{C}_{22} \dt \| \tilde{x}^{n+1} \|_2^2
    + \tilde{C}_{23} \dt  \| \tilde{x}^n \|_2^2
    + \frac{\tilde{C}_{21}}{2}  \dt \| \tilde{x}^{n-1} \|_2^2  \nonumber
\\
  &&
  +  \tilde{C}_{24} \dt^3 ( \| D_h \tilde{x}^{n+1} \|_2^2 +  \| D_h \tilde{x}^n \|_2^2 )
  + \dt ( \| \tau^n \|_2^2 + ( \dt^4 + h^4 )^2  )  ,
  \label{convergence-8-1}
\end{eqnarray}
with $\tilde{C}_{21} = \tilde{C}_{14} + \tilde{C}_{14}^2 ( \tilde{C}^* )^2$,
$\tilde{C}_{22} = \frac{\tilde{C}_{14}}{2} + \frac{\tilde{C}_6}{2}  + 2$, $\tilde{C}_{23} = \tilde{C}_4 + \frac12 + \frac{\tilde{C}_{14}}{2} + \frac{9 \tilde{C}_{21}}{2}$, $\tilde{C}_{24} = \tilde{C}_{20} + 4 (\epsilon_0^*)^{-2}$. Subsequently, a summation in time gives
\begin{eqnarray}
   &&
   \Bigl\langle \frac{ (S_h (x^n , x^{n-1} )  )^{m-1} }{m ( f_0(X) )^{m-2} } ,
    ( \tilde{x}^{n+1} )^2 \Bigr\rangle
    +  \frac{1}{( \tilde{C}^*)^2 }  \dt \sum_{k=0}^n \|  D_h \tilde{x}^{k+\nicefrac12}  \|_2^2  \nonumber
\\
  &&
    + \dt^2 \left( A_0 \| D_h \tilde{x}^{n+1} \|_2^2
    + \langle  \gamma^{n+\nicefrac12} , ( D_h \tilde{x}^{n+1} )^2  \rangle \right)    \nonumber
\\
  &\le&
      \tilde{C}_{25} \dt \sum_{k=0}^{n+1} \| \tilde{x}^k \|_2^2
  +  2 \tilde{C}_{24} \dt^3 \sum_{k=0}^{n+1} \| D_h \tilde{x}^k \|_2^2
  + C ( \dt^4 + h^4 )^2   ,
  \label{convergence-8-2}
\end{eqnarray}
with $\tilde{C}_{25} = \tilde{C}_{22}  + \tilde{C}_{23} + \frac{\tilde{C}_{21}}{2}$. Meanwhile, by the definition of $\gamma^{n+\nicefrac12}$ ~\eqref{convergence-7-12}, we have
\begin{eqnarray}
  \|  \gamma^{n+\nicefrac12} \|_\infty \le \frac{2 C^* +1}{(\epsilon_0^*)^3} , \quad
  \mbox{by~\eqref{assumption:constructed bound}, \eqref{prelim bound-3}, \eqref{convergence-rough-9-1} } .   \label{convergence-8-3}
  \end{eqnarray}
In turn, by taking $A_0 = (2 C^* +1) (\epsilon_0^*)^{-3} +1$, and making use of the inequality~\eqref{convergence-rough-8-2}, we obtain
\begin{eqnarray}
   &&
   \tilde{C}_8 \| \tilde{x}^{n+1} \|_2^2
    +  \frac{1}{( \tilde{C}^*)^2 }  \dt \sum_{k=0}^n \|  D_h \tilde{x}^{k+\nicefrac12}  \|_2^2
    + \dt^2  \| D_h \tilde{x}^{n+1} \|_2^2
      \nonumber
\\
  &\le&
      \tilde{C}_{25} \dt \sum_{k=0}^{n+1} \| \tilde{x}^k \|_2^2
  +  2 \tilde{C}_{24} \dt^3 \sum_{k=0}^{n+1} \| D_h \tilde{x}^k \|_2^2
  + C ( \dt^4 + h^4 )^2   .
  \label{convergence-8-4}
\end{eqnarray}
Therefore, an application of discrete Gronwall inequality (in the integral form) leads to the desired higher order convergence estimate
\begin{equation}
 \| \tilde{x}^{n+1} \|_2 +  \Bigl( (\tilde{C}^*)^{-2} \dt   \sum_{m=0}^{n} \| \frac12 D_h ( \tilde{x}^{m+1} + \tilde{x}^m ) \|_2^2 \Bigr)^{1/2}  \le \hat{C}_4 ( \dt^4 + h^4 ) .
	\label{convergence-8-5}
	\end{equation}
This completes the refined error estimate.

\subsection{Recovery of the a-priori assumption~\eqref{a priori-1}}

With the higher order error estimate~\eqref{convergence-8-5} at hand, we conclude that the a-priori assumption in~\eqref{a priori-1} is satisfied at the next time step $t^{n+1}$, since $\hat{C}_4$ takes a following form:
\begin{eqnarray}
  \hat{C}_4 \le C \exp
  \Big( \frac{(\tilde{C}_{25} + 2 \tilde{C}_{24} ) t^{n+1} }{\min (\tilde{C}_8, 1 ) }  \Big)
  \le \hat{C} := C \exp
  \Big( \frac{(\tilde{C}_{25} + 2 \tilde{C}_{24} ) T }{\min (\tilde{C}_8, 1 ) }  \Big) .
   \label{a priori-2}
\end{eqnarray}
We also notice that $\tilde{C}_8$, $\tilde{C}_{24}$ and $\tilde{C}_{25}$ are independent of ${\cal C}$. Therefore, the a-priori assumption in~\eqref{a priori-1} is satisfied, so that an induction analysis could be applied. This finishes the higher order convergence analysis.

Finally, the convergence estimate~\eqref{convergence-0} is a direct consequence of~\eqref{convergence-8-5}, combined with the definition~\eqref{consistency-1} of the  constructed approximate solution $W$. This completes the proof of Theorem~\ref{thm:convergence}.

\section{Convergence analysis of Newton's iteration}\label{sec:Newton}
In this section,
we prove the convergence of damped Newton's iteration \eqref{scheme-2nd-Newton} in the convex set $\mathcal{Q}$, based on $\emph{self-concordant}$ \cite{S.Boyd(2004),Y. Nesterov(1994)}.
The definition of $\emph{self-concordant}$ is given as following:
\begin{defn}\label{def:self-concordant}
Let $\mathcal{G}$ be a finite-dimensional real vector space, $\mathcal{Q}$ be an open nonempty convex subset of $\mathcal{G}$, $\Lambda : \mathcal{Q}\rightarrow\mathbb{R}$ be a function, $a>0$. $\Lambda$ is called self-concordant on $\mathcal{Q}$ with the parameter value $a$, if $\Lambda \in C^3$ is a convex function on $\mathcal{Q}$, and, for all $x\in \mathcal{Q}$ and all $u\in\mathcal{G}$, the following inequality holds:
$$|D^{3}\Lambda(x)[u,u,u]|\leq 2a^{-1/2}(D^2\Lambda(x)[u,u])^{3/2}$$
($D^{k}\Lambda(x)[u_1,\cdots,u_k]$ henceforth denotes the value of the kth differential of $\Lambda$ taken at $x$ along the collection of directions $u_1,\cdots,u_k$) \cite{S.Boyd(2004),Y. Nesterov(1994)}. 
\end{defn}
 The $\emph{self-concordant}$ function has two typical characteristic  \cite{S.Boyd(2004),Y. Nesterov(1994)}:
 \begin{itemize}
 \item Linear and (convex) quadratic functions  are evidently self-concordant, since they have zero third derivative.
 \item  A function $f:\mathcal{R}^n\rightarrow R$
is self-concordant if it is self-concordant along every line in its domain. \end{itemize}

\begin{thm}
Suppose $f_0(X)\in\mathcal{E}_{N}$ is the initial distribution with a positive lower bound for $X\in \mathcal{Q}$ and $a:=(h\min\limits_{0\leq i\leq M}f_0(X_i))/2C_{Newton}^2$ with a positive constant $C_{Newton}$,  then $F(\hat{x})$, defined in \eqref{solvability-2-1}-\eqref{solvability-2-4}, is a self-concordant function  and Newton's iteration \eqref{scheme-2nd-Newton} is convergent in $\mathcal{Q}$.
\label{thm:Newton}
\end{thm}
\noindent\textbf{Proof.} Since  linear and quadratic functions  have zero third derivative, $F_1(\hat{x})$ and $F_3(\hat{x})$ are  self-concordant.  We  just need to prove $F_2(\hat{x})$ and $F_4(\hat{x})$ are   self-concordant  along every line in $\mathcal{Q}$.

Suppose  $ x^n\in\mathcal{Q}$ and let  $$\xi_{i-\frac{1}{2}}:=\frac{D_h x^n_{i-\frac{1}{2}}}{1+D_h \hat{x}_{i-\frac{1}{2}}}, i=1,\cdots,M.$$
Then $\xi_{i-\frac{1}{2}}>0,\  i=1,\cdots,M$.
For $\forall i=1,\cdots,M$,  we can obtain  
\begin{equation*}
\begin{split}
\frac{\partial F_2(\hat{x}_i)}{\partial \hat{x}_i}&=\frac{f_0(X_{i-\frac{1}{2}})}{1+D_h \hat{x}_{i-\frac{1}{2}}-D_h x^n_{i-\frac{1}{2}}}\ln(\xi_{i-\frac{1}{2}})\\
&-\frac{f_0(X_{i+\frac{1}{2}})}{1+D_h \hat{x}_{i+\frac{1}{2}}-D_h x^n_{i+\frac{1}{2}}}\ln(\xi_{i+\frac{1}{2}}),
\end{split}
\end{equation*}

\begin{equation*}
\begin{split}
\frac{\partial^2 F_2(\hat{x}_i)}{\partial\hat{x}_{i}^2}=&-\frac{f_0(X_{i-\frac{1}{2}})}{h\Big(1+D_h \hat{x}_{i-\frac{1}{2}}-D_h x^n_{i-\frac{1}{2}}\Big)^2}[(1-\xi_{i-\frac{1}{2}})+\ln(\xi_{i-\frac{1}{2}})]\\
&-\frac{f_0(X_{i+\frac{1}{2}})}{h\Big(1+D_h \hat{x}_{i+\frac{1}{2}}-D_h x^n_{i+\frac{1}{2}}\Big)^2}[(1-\xi_{i+\frac{1}{2}})+\ln(\xi_{i+\frac{1}{2}})],
\end{split}
\end{equation*}
\begin{equation*}
\begin{split}
\Big|\frac{\partial^3 F_2(\hat{x}_i)}{\partial\hat{x}_i^3}\Big|&=\Big|\frac{2f_0(X_{i-\frac{1}{2})}}{h^2\Big(1+D_h \hat{x}_{i-\frac{1}{2}}-D_h x^n_{i-\frac{1}{2}}\Big)^3}\left[(1-\xi_{i-\frac{1}{2}})+\ln(\xi_{i-\frac{1}{2}})+\frac{1}{2}(1-\xi_{i-\frac{1}{2}})^2\right]\\
&-\frac{2f_0(X_{i+\frac{1}{2})}}{h^2\Big(1+D_h \hat{x}_{i+\frac{1}{2}}-D_h x^n_{i+\frac{1}{2}}\Big)^3}\left[(1-\xi_{i+\frac{1}{2}})+\ln(\xi_{i+\frac{1}{2}})+\frac{1}{2}(1-\xi_{i+\frac{1}{2}})^2\right]\Big|\\
&\leq \Big|\frac{2f_0(X_{i-\frac{1}{2})}}{h^2\Big(1+D_h \hat{x}_{i-\frac{1}{2}}-D_h x^n_{i-\frac{1}{2}}\Big)^3}\left[(1-\xi_{i-\frac{1}{2}})+\ln(\xi_{i-\frac{1}{2}})+\frac{1}{2}(1-\xi_{i-\frac{1}{2}})^2\right]\\
&+\frac{2f_0(X_{i+\frac{1}{2})}}{h^2\Big(1+D_h \hat{x}_{i+\frac{1}{2}}-D_h x^n_{i+\frac{1}{2}}\Big)^3}\left[(1-\xi_{i+\frac{1}{2}})+\ln(\xi_{i+\frac{1}{2}})+\frac{1}{2}(1-\xi_{i+\frac{1}{2}})^2\right]\Big|.
\end{split}
\end{equation*}
Note that  $$\ln t= \ln (1+(t-1))=(t-1)-\frac{1}{2}(t-1)^2+O(t-1)^3, \ \forall t>0,$$  
hence there exists a constant $C_{Newton}>0$ such that 
\begin{equation*}
\left[(1-\xi_{i-\frac{1}{2}})+\ln(\xi_{i-\frac{1}{2}})+\frac{1}{2}(1-\xi_{i-\frac{1}{2}})^2\right]^2\leq C_{Newton}\left[-(1-\xi_{i-\frac{1}{2}})-\ln(\xi_{i-\frac{1}{2}})\right]^3.
\end{equation*}
If  the parameter $a:=(h\min\limits_{0\leq i\leq M}f_0(X_i))/2C_{Newton}^2$, we obtain 
\begin{equation}\label{equ:selfcon}
\begin{split}
&\left|\frac{2f_0(X_{i-\frac{1}{2})}}{h^2\Big(1+D_h \hat{x}_{i-\frac{1}{2}}-D_h x^n_{i-\frac{1}{2}}\Big)^3}\left[(1-\xi_{i-\frac{1}{2}})+\ln(\xi_{i-\frac{1}{2}})+\frac{1}{2}(1-\xi_{i-\frac{1}{2}})^2\right]\right|^2\\
&\leq 2a^{-\frac{1}{2}}\left(-\frac{f_0(X_{i-\frac{1}{2}})}{h\Big(1+D_h \hat{x}_{i-\frac{1}{2}}-D_h x^n_{i-\frac{1}{2}}\Big)^2}\left[(1-\xi_{i-\frac{1}{2}})+\ln(\xi_{i-\frac{1}{2}})\right]\right)^3, \forall i=1,\cdots, M.
\end{split}
\end{equation}
So $F_2(\hat{x})$  is self-concordant.   By the similar method, $F_4(\hat{x})$ is also  self-concordant. Based on {\bf{Theorem 2.2.3}} in \cite{Y. Nesterov(1994)}, Newton's iteration is convergent in $\mathcal{Q}$.$\hfill\Box$

\section{The numerical results}
\label{sec:numResults}
In this section, we present an example with a positive state   to demonstrate the convergence rate of the numerical scheme.

Before that, we define the error of a numerical solution measured  in the $\mathcal{L}^2$ and $\mathcal{L}^{\infty}$ norms as:
\begin{equation}\label{L2}
\|e_h\|_2^2=\frac{1}{2}\left(e_{h_0}^2 h_{x_{0}}+\sum\limits_{i=1}^{M-1}e_{h_i}^2 h_{x_i}+e_{h_M}^2 h_{x_M}\right),
\end{equation}
and
\begin{equation}\label{Linf}
\|e_h\|_{\infty}=\max\limits_{0\leq i\leq M}\{|e_{h_i}|\},
\end{equation}
 where $e_h=(e_{h_0},e_{h_1},\cdots,e_{h_M})$ and
for the error of the density  $e_h^f:=f-f_h$,
 \begin{equation*}
 h_{x_i}=x_{i+1}-x_{i-1}, \ \ 1\leq i \leq M-1; \ \ \
 h_{x_0}=x_{1}-x_{0}; \ \
 h_{x_M}=x_{M}-x_{M-1}, \end{equation*}
and for  the error of the  trajectory $e_h^x:=x-x_h$,
 \begin{equation}
h_{x_i}=2h,  \ \ 1\leq i\leq M-1, \ \ \ h_{x_0}=h_{x_M}=h , \nonumber
\end{equation}
 where $h$ is the spatial step.

 Consider the problem \eqref{eqcm}-\eqref{eqbc} in dimension one with a smooth positive initial data
\begin{equation}\label{eqiniEx}
f_0(x)=0.5-(x-0.5)^2, x\in\Omega:=[0,1].
\end{equation}
Firstly, the trajectory equation \eqref{eqtra} with the initial and boundary condition \eqref{eqtrabou}-\eqref{eqtraini} can be solved by the fully discrete scheme \eqref{scheme-2nd-1}. Subsequently, the density function $f$ in \eqref{equ:conservationL} can be approximated by \eqref{eqDen}. The reference ``exact" solution is obtained numerically on a much finer mesh with $h=\frac{1}{10000},\  \tau=h$.  We choose $a=(h\min\limits_{0\leq i\leq M}f_0(X_i))/2$ in Theorem \ref{thm:Newton}.  
Table \ref{table:Convergence}  shows the second order convergence  for  density $f$ and trajectory $x$ in the   $\mathcal{L}^2$ and  $\mathcal{L}^{\infty}$ norm with both $m=\frac{5}{3}$ and $m=2$ at time $t=0.05$. The results  verify  the optimal convergence rate of the numerical scheme.  

\begin{center}
\scriptsize
\begin{threeparttable}[b]
\setlength{\abovecaptionskip}{10pt}
\setlength{\belowcaptionskip}{3pt}
\caption{\scriptsize  Convergence  rate of solution $f$ and trajectory $x$ at time $t=0.05$}
\label{table:Convergence}
\begin{tabular}{@{ } l c c c c c  c c c c}
\hline
\multicolumn{1}{l}{}
&\multicolumn{8}{c}{$m=5/3$}\\\hline
 $h$    &$\tau$ &$\|e_h^f\|_2$   & Order &$\|e_h^f\|_{\infty}$   & Order &$ \|e_h^x\|_2 $   & Order  &$ \|e_h^x\|_{\infty} $   & Order \\\hline
1/200 &1/200&1.506e-04 && 3.277e-04& &7.593e-05&&7.844e-05&\\\hline
1/400 &1/400&3.620e-05 &2.056&8.421e-05&1.960&1.871e-05&2.021&1.934e-05&2.020\\\hline
 1/800 &1/800 &8.495e-06 &2.092&2.033e-05&2.050&4.464e-06&2.067&4.617e-06&2.066\\\hline
1/1600 &1/1600 & 1.887e-06&2.170&4.695e-06&2.114  &1.000e-06&2.158&1.036e-06&2.156  \\\hline

\hline
\multicolumn{1}{l}{}
&\multicolumn{8}{c}{$m=2$}\\\hline
 $h$    &$\tau$ &$\|e_h^f\|_2$   & Order &$\|e_h^f\|_{\infty}$   & Order &$ \|e_h^x\|_2 $   & Order  &$ \|e_h^x\|_{\infty} $   & Order \\\hline
1/200 &1/200 & 1.502e-04&  &3.279e-04& &7.642e-05& & 7.902e-05 \\\hline
1/400 &1/400 &3.599e-05 & 2.061&8.370e-05&1.970&1.873e-05&2.028&1.938e-05&2.028\\\hline
1/800 &1/800 &8.431e-06 &2.094&2.005e-05&2.061&4.458e-06&2.071&4.615e-06&2.070\\\hline
1/1600 &1/1600 &1.853e-06 &2.186&4.563e-06&2.136&9.871e-07&2.175&1.024e-06&2.172\\\hline
\end{tabular}
\begin{tablenotes}
     \scriptsize

                \item[1]    $\tau$ is the time step and  $h$ is the space step.
        \end{tablenotes}
\end{threeparttable}
\end{center}

 \section{Concluding remarks} \label{sec: conclusion} 

The porous medium equation, based on an energetic variational approach, is taken into consideration. We develop and analyze a second order accurate numerical scheme. The unique solvability, energy stability are proved with the help of the convexity analysis. In addition, we provide a detailed convergence analysis for the proposed numerical scheme, which is accomplished by a higher order asymptotic expansion of the numerical solution, combined with two step error estimates: a rough estimate is to control the highly nonlinear term in a discrete $W^{1,\infty}$ norm, and a refined estimate is o derive the optimal error order. The convergence of the Newton's iteration is analyzed as well. Some numerical examples are presented. 
 

\section*{Acknowledgments}
This work is supported in part by the Grants NSFC 11671098, 11331004, 91630309, a 111 Project B08018 (W. Chen), NSFC 11901109 (C. Duan), NSF-DMS 1759535, NSF-DMS 1759536 (C. Liu), NSF DMS-1418689 (C. Wang), NSFC 11271281 (X. Yue). C. Wang also thanks the Key Laboratory of Mathematics for Nonlinear Sciences, Fudan University, for support during his visit.

\appendix

\section{Proof of Lemma~\ref{lem:convexity-1}} \label{sec:proof lemma 2}

A direct calculation gives
\begin{eqnarray}
  q'_1 (x) = - \frac{\frac{1}{x} (x - x_0) - ( \ln x - \ln x_0)}{(x- x_0)^2} > 0 ,
  \label{lem 2-1}
\end{eqnarray}
in which the convexity of $- \ln x$ (for $x>0$) has been applied:
\begin{eqnarray}
  - \frac{1}{x} (x - x_0) + ( \ln x - \ln x_0)  > 0 .  \label{lem 2-2}
\end{eqnarray}
Meanwhile, a detailed Taylor expansion leads to
\begin{eqnarray}
    \ln x - \ln x_0 = \frac{1}{x} (x - x_0) + \frac{1}{2 \zeta^2} (x - x_0)^2 ,
    \quad \mbox{with $\zeta$ between $x_0$ and $x$},  \label{lem 2-3}
\end{eqnarray}
which in turn implies that
\begin{eqnarray}
  q'_1 (x) = - \frac{\frac{1}{x} (x - x_0) - ( \ln x - \ln x_0)}{(x- x_0)^2}
  = \frac{1}{2 \zeta^2} ,   \quad \mbox{with $\zeta$ between $x_0$ and $x$} .
   \label{lem 2-4}
\end{eqnarray}
Therefore, an application of the intermediate value theorem indicates that
\begin{eqnarray}
  &&
  q_1 (y) - q_1 (x) = q'_1 (\eta) (y-x) ,  \, \, \mbox{with $\eta$ between $x$ and $y$} ,
   \label{lem 2-5-1}
\\
  &&
  q'_1 (\eta) = \frac{1}{2 \zeta_\eta^2} ,   \quad \mbox{with $\zeta_\eta$ between $x_0$ and $\eta$}  .  \label{lem 2-5-2}
\end{eqnarray}
Of course, a careful analysis implies that
\begin{eqnarray}
  q'_1 (\eta) = \frac{1}{2 \zeta_\eta^2} ,   \quad \mbox{$q'_1 (\eta)$ is between} \, \,
  \frac{1}{2 y^2}, \, \frac{1}{2 x^2} , \, \mbox{and} \, \frac{1}{2 x_0^2}, \, \,
   \forall x>0, \, y>0  .   \label{lem 2-5-3}
\end{eqnarray}
This completes the proof of~\eqref{q convexity-2}.

A further calculation gives
\begin{eqnarray}
  q''_1 (x) = - \frac{-\frac{1}{x^2} (x - x_0)^3
  - 2 (x- x_0) \Big( \frac{1}{x} (x - x_0) - ( \ln x - \ln x_0) \Big) }{(x- x_0)^4} \le 0 ,
  \label{lem 2-6-1}
\end{eqnarray}
for any $x >0$, in which a higher order Taylor expansion has been applied:
\begin{eqnarray}
    \ln x - \ln x_0 = \frac{1}{x} (x - x_0) + \frac{1}{2 x^2} (x - x_0)^2
    + \frac{1}{3 \zeta^3} (x - x_0)^3,
    \quad \mbox{with $\zeta$ between $x_0$ and $x$} .  \label{lem 2-6-2}
\end{eqnarray}
As a direct consequence, by an introduction of $q_2 (x) :=  \frac{q_1 (y) - q_1 (x)}{y -x}$ for a fixed $y >0$, we get
\begin{eqnarray}
  &&
   q'_2 (x) = \frac{-q'_1 (x) (y - x) + ( q_1 (y) - q_1 (x) )}{(y-x)^2} \le 0 ,
  \label{lem 2-7}
\\
  &&  \mbox{since $q_1 (y)$ is concave:} \, \, \,
  - q'_1 (x) (y - x) + ( q_1 (y) - q_1 (x) ) \le 0 ,  \, \, \, \forall y>0 , \, x > 0 .
  \nonumber
\end{eqnarray}
This completes the proof of Lemma~\ref{lem:convexity-1}.

 {\footnotesize
\bibliographystyle{unsrt}

\begin{thebibliography}{0}
\bibitem{D.G. Aronson(1969)}
D. G. Aronson, Regularity properties of flows through porous media, SIAM J. Appl. Math. 17 (1969) 461-467.

\bibitem{D.G.Aronson(1983)}
D. G. Aronson, L. A. Caffarelli, S. Kamin, How an initially stationary interface begins to move in porous medium flow, SIAM J. Math. Anal. 14 (4) (1983) 639-658.

\bibitem{A. Baskaran(2013)}
A. Baskaran, J. L. Lowengrub, C. Wang, S. M. Wise, Convergence analysis of a second order convex splitting scheme for the modified phase field crystal equation, SIAM J. Numer. Anal. 51 (2013), 2851-2873.

\bibitem{S.Boyd(2004)}
S. Boyd and L. Vandenberghe, Convex optimization, Cambridge Univ. Press (2004).

\bibitem{E. DiBenedetto(1984)}
E. DiBenedetto, D. Hoff,  An interface tracking algorithm for the porous medium equation, Trans. Am. Math. Soc. 284 (1984)  463-500.

\bibitem{C. Duan(2018)}
 C. Duan, C. Liu, C. Wang, X. Yue,  Numerical methods for Porous Medium Equation by an Energetic Variational Approach, J. Comput. Phys. 385 (2019) 13-32.
 
\bibitem{C.H.Duan(2019)}
C. Duan, C. Liu , C. Wang, X.  Yue,  Convergence Analysis of a Numerical Scheme for the Porous Medium Equation by an Energetic Variational Approach, Numer. Math-Theory Me. 13 (2020).

\bibitem{W. E(1995)}
W. E, J. G. Liu, Projection method I: Convergence and numerical boundary layers, SIAM J. Numer. Anal. 32 (1995) 1017-1057.

\bibitem{W. E(2002)}
W. E, J. G. Liu, Projection method. III. Spatial discretization on the staggered grid, Math. Comp. 71 (2002) 27-47.

\bibitem{J. L. Graveleau(1971)}
J. L. Graveleau,  P. Jamet, A finite difference approach to some degenerate nonlinear parabolic equations, SIAM J.  Appl.  Math. 20 (1971) 199-223.

\bibitem{Z. Guan(2017)}
Z. Guan, J. L. Lowengrub, C. Wang, Convergence analysis for second order accurate schemes for the periodic nonlocal Allen-Cahn and Cahn-Hilliard equations, Math. Model. Appl. Sci. 40 (2017) 6836-6863.

\bibitem{Z. Guan(2014)}
Z. Guan, C. Wang, S. M. Wise, A convergent convex splitting scheme for the periodic nonlocal Cahn-Hilliard equation, Numer. Math. 128 (2014) 377-406.

\bibitem{S. Jin(1998)}
S. Jin, L. Pareschi,  G. Toscani, Diffusive relaxation schemes for multi-scale discrete-velocity kinetic equations,  SIAM  J. Numer. Anal. 35  (6) (1998) 2405-2439.

\bibitem{A.S. Kalasnikov(1967)}
A. S. Kala\v{s}nikov, Formation of singularities in solutions of the equation of nonstationary filtration, \v{Z}. Vy\v{c}isl. Mat. Mat. Fiz. 7 (1967) 440-444.

\bibitem{Y. Nesterov(1994)}
Y. Nesterov,  A. Nemirovskii,  \emph{Interior-point polynomial algorithms in convex programming}. SIAM  \textbf{13} (1994).

\bibitem{C. Ngo(2017)}
C. Ngo, W. Z.  Huang,  A study on moving mesh finite element solution of the porous medium equation,  J. Compu. Phys. 331 (2017) 357-380.


 \bibitem{O.A. Oleinik(1958)}
O. A. Ole\v{\i}nik, A. S. Kala\v{s}inkov,  Y. \v{C}\v{z}ou, The Cauchy problem and boundary problems for equations of the type of non-stationary filtration, Izv. Akad. Nauk SSSR, Ser. Mat. 22 (1958) 667-704.

\bibitem{R. Samelson(2003)}
R. Samelson, R. Temam, C. Wang, S. Wang, Surface pressure Poisson equation formulation of the primitive equations: Numerical schemes, SIAM J. Numer. Anal. 41 (2003), 1163-1194.

\bibitem{S. Shmarev(2003)}
S. I. Shmarev, Interfaces in multidimensional diffusion equations with absorption terms, Nonlinear Anal. 53 (2003) 791-828.


\bibitem{S. Shmarev(2005)}
S. Shmarev, Interfaces in solutions of diffusion-absorption equations in arbitrary space dimension, in: Trends in Partial Differential Equations of Mathematical Physics, in: Progr. Nonlinear Differential Equations Appl.  Birkh\"auser, Basel, 2005, pp. 257-273.

 \bibitem{J. L. Vazquez(2007)}
J. L. V\'azquez, The Porous Medium Equation, Oxford University Press, Oxford, 2007.

\bibitem{C. Wang(2004)}
C. Wang, J. G. Liu, H. Johnston, Analysis of a fourth order finite difference method for incompressible Boussinesq equations, Numer. Math.  97 (2004) 555-594.

\bibitem{C. Wang(2002)}
C. Wang, J. G. Liu, Analysis of finite difference schemes for unsteady Navier-Stokes equations in vorticity formulation, Numer. Math.  91 (2002) 543-576.

\bibitem{C. Wang(2000)}
C. Wang, J. G. Liu, Convergence of gauge method for incompressible flow, Math. Comp.  69 (2000) 1385-1407.

\bibitem{L. Wang(2015)}
L. Wang, W. Chen, C. Wang, An energy-conserving second order numerical scheme for nonlinear hyperbolic equation with an exponential nonlinear term, J. Comput. Appl. Math.  280 (2015) 347-366.

 \bibitem{M. Westdickenberg(2010)}
M. Westdickenberg, J. Wilkening, Variational particle schemes for the porous medium equation and for the system of isentropic Euler equations, ESAIM: M2AN. 44 (1) (2010) 133-166.

\bibitem{Q. Zhang(2009)}
Q. Zhang,  Z. L. Wu, Numerical simulation for porous medium equation by local discontinuous Galerkin finite element method, J. Sci. Comput. 38 (2) (2009) 127-148.

\end{thebibliography}

}

\end{document}